\documentclass[10pt,a4paper,titling,fleqn]{article} 
\usepackage[paper=a4paper,left=25mm,right=25mm,top=25mm,bottom=25mm]{geometry}

\usepackage[utf8]{inputenc}
\usepackage[english]{babel}
\usepackage[T1]{fontenc}
\usepackage{lmodern}
\usepackage{amsmath, amstext, amsfonts, mathrsfs}
\usepackage{amsfonts}
\usepackage{amssymb}
\usepackage{dsfont}
\usepackage{mathrsfs}
\usepackage{graphicx}
\usepackage{epsfig}
\usepackage[authoryear]{natbib}
\usepackage{color}
\usepackage{bm}
\usepackage{verbatim}
\usepackage{stmaryrd}
\usepackage{appendix}
\usepackage[colorlinks=false,linkcolor=darkblue]{hyperref}
\usepackage{bbm}			
\usepackage{enumerate} 
\usepackage{nicefrac}
\usepackage{xfrac}
\usepackage{units} 
\usepackage{mathtools}
\usepackage{pdfcomment}
\usepackage{url}
\usepackage{tikz}
\usepackage{cancel} 
\usepackage[labelfont={bf,sf},font={small},%
  labelsep=space]{caption}
\usepackage{acronym}
\usepackage{mathrsfs}
\usepackage{fancyhdr}
\usepackage[T1]{fontenc}
\usepackage{caption, booktabs}%für Tabellen
\usepackage{textcomp}
\usepackage{multirow}
\usepackage{tikz}
\usetikzlibrary{arrows.meta}
\usetikzlibrary{positioning,arrows}
\usetikzlibrary{matrix}
\usepackage{hologo}
\usepackage{xcolor}
\usepackage{forest}
\usetikzlibrary{matrix,fit}
\tikzset{%
  mleftdelimiter/.style={inner ysep=0pt, inner xsep=1ex,left delimiter=\{,label={[label distance=3mm]left:#1}}
}

\usepackage[Bjarne]{fncychap} 
\usepackage{setspace}
\usepackage{etex,m-pictex,color}
\usepackage[squaren]{SIunits}
\usepackage{overpic}
\usepackage{subcaption}

\definecolor{light-gray}{gray}{0.95}
\definecolor{darkblue}{rgb}{0,0,.5}

\newcommand{\C}{\mathbb{C}}

\newcommand{\E}{\mathbb{E}}

\newcommand{\N}{\mathbb{N}}

\newcommand{\R}{\mathbb{R}}

\newcommand{\uu}{\bold{u}}
\newcommand{\UU}{\bold{U}}
\newcommand{\VV}{\bold{V}}
\newcommand{\WW}{\bold{W}}

\newcommand{\SSS}{\bold{S}}
\newcommand{\XX}{\mathbf{X}}
\newcommand{\YY}{\mathbf{Y}}
\newcommand{\ZZ}{\bold{Z}}
\newcommand{\ttt}{\bold{t}}
\newcommand{\xx}{\bold{x}}

\newcommand{\cC}{\mathcal{C}}

\newcommand{\cE}{\mathcal{E}}

\newcommand{\cS}{\mathcal{S}}
\newcommand{\cU}{\mathcal{U}}

\providecommand{\keywords}[1]{\textbf{Keywords } #1}
\newcommand{\eqd}{\stackrel{\mathrm{d}}=}

\newcommand{\1}{\mathds{1}}

\newcommand{\conv}{{\scalebox{1.5}{\raisebox{-0.2ex}{$\ast$}}}}

\newcommand{\de}{\mathrm{\,d}}

\newcommand{\Ran}{\mathsf{Ran}}

\newcommand{\rank}{\mathsf{rank}}

\newcommand{\var}{\mathrm{Var}}

\newcommand{\Var}{\mathrm{Var}}

\newcommand{\JA}[1]{{\color{orange} #1}}

\DeclareMathOperator{\v@r}{V@R}

\DeclareMathOperator{\av@r}{AV@R}

\newtheorem{theorem}{Theorem}[section]
\newtheorem{proposition}[theorem]{Proposition}
\newtheorem{corollary}[theorem]{Corollary}

\newtheorem{lemma}[theorem]{Lemma}
\newtheorem{example}[theorem]{Example}
\newtheorem{remark}[theorem]{Remark}

\newenvironment{proof}[1][{Proof:}]{\begin{trivlist}
\item[\hskip \labelsep {\bfseries #1}]}{\hfill$\blacksquare$ \end{trivlist}}

\selectfont

\author{%
  Jonathan Ansari\textsuperscript{1} and 
  Sebastian Fuchs\textsuperscript{1} \\
  \\
  \textsuperscript{1} University of Salzburg, Austria \\
  \texttt{\{jonathan.ansari, sebastian.fuchs\}@plus.ac.at}
}

\title{On continuity of Chatterjee's rank correlation\\ and related dependence measures}

\begin{document}

\maketitle

\begin{abstract}
While measures of concordance---such as Spearman's rho, Kendall's tau, and Blomqvist's beta---are continuous with respect to weak convergence, Chatterjee’s rank correlation \(\xi\) recently introduced by \cite{chatterjee2021} does not share this property, causing drawbacks in statistical inference as pointed out by \cite{buecher2024}.
As we show in this paper, \(\xi\) is instead weakly continuous with respect to conditionally independent copies---the Markov products. To establish weak continuity of Markov products, we provide several sufficient conditions, including copula-based criteria and conditions relying on the concept of conditional weak convergence from \cite{Sweeting_1989}. 
As a consequence, we obtain continuity results for \(\xi\) and related dependence measures and verify their continuity in standard models such as multivariate elliptical and $\ell_1$-norm symmetric distributions.

% \PACS{PACS code1 \and PACS code2 \and more}
% \subclass{MSC code1 \and MSC code2 \and more}
%
%\subclass{60 E 15 \and  62 P 05\and 91 B 28\and 91 B 30}

\end{abstract}

% Titel

\maketitle

%%%%%%%%%%%%%%%%%%%%%%%%%%%%%%%%%%%%%%%%%%%%%%%%%%%%%%%%%%%%%%%%%%%%%%%%%%%%%%%%%%%%%%
\section{Introduction}
\label{Sec.Continuity}

In recent years, dependence measures that are able to characterize independence and perfect dependence have attracted a lot of attention in the statistics literature; see \cite{deb2020b, wiesel2022, chatterjee2020, strothmann2022}.
The certainly most famous one is Chatterjee's rank correlation whose population version is defined for a non-degenerate random variable \(Y\) and a random vector \(\XX\) by
\begin{align}\label{defxi}
    \xi(Y,\XX)
	:= \frac{\int_{\mathbb{R}} {\Var} (P(Y \geq y \, | \, \XX)) \; \mathrm{d} P^{Y}(y)}
					{\int_{\mathbb{R}} {\Var} (\mathds{1}_{\{Y \geq y\}}) \; \mathrm{d} P^{Y}(y)};
\end{align}
see \cite{chatterjee2021,Gamboa-Klein-Lagnoux-2018,siburg2013}. The measure \(\xi\) has the fundamental properties that \(\xi(Y,\XX) \in [0,1],\) where \(\xi(Y,\XX) = 0\) if and only if \(\XX\) and \(Y\) are independent, and \(\xi(Y,\XX) = 1\) if and only if \(Y\) \emph{perfectly depends} on \(\XX,\) i.e., there exists a measurable (not necessarily increasing or decreasing) function \(f\) such that \(Y = f(\XX)\) almost surely. Further, \(\xi\) exhibits several additional properties such as information monotonicity or characterization of conditional independence; see \citet[Lemma 11.2]{chatterjee2021}. 
However, the behavior of \(\xi\) within the interval \((0,1)\) is not yet fully understood. For example, one might ask how the value of \(\xi\) changes when a small perturbation is introduced into the response variable.
In the additive error model 
\begin{align}\label{eqadderrmod}
    Y = X + \sigma\varepsilon, \quad \sigma > 0, ~ \varepsilon \text{ and } X \text{ independent},
\end{align}
one expects that \(\xi(Y,X)\) is continuous in \(\sigma.\)
It is straightforward to verify that the bivariate random vector \((Y,X)\) in \eqref{eqadderrmod} is continuous in \(\sigma\) with respect to weak convergence. However, unlike %the Pearson correlation and 
concordance measures---such as Kendall's tau, Spearman's rho or Blomqvist's beta---Chatterjee's rank correlation fails to be weakly continuous, as the following example demonstrates.

%In is well-known that the Pearson correlation and measures of concordance such as Kendall's tau, Spearman's rho, or Blomqvist's beta are continuous with respect to weak convergence. However, as the following example shows, Chatterjee's rank correlation does not share this property.

\begin{example}[\(\xi\) is not continuous w.r.t. weak convergence of \((Y_n,X_n)\) to \((Y,X)\)]~~\label{Cont.ExWC}\\
Assume that the joint distribution of the bivariate random vector \((Y_n,X_n)\) follows some so-called shuffle-of-min copula; see \cite{Mikusinski-1992} for the definition. Then, \(Y_n\) and \(X_n\) are uniformly distributed on \((0,1)\), and \(Y_n=f_n(X_n)\) almost surely for some measurable function \(f_n\).
Since shuffles of min are dense in the class of doubly stochastic measures on \((0,1)^2\) with respect to weak convergence, \((Y_n,X_n)\) may converge in distribution, for example, to a random vector \((Y,X)\) with independent components; see \citet[Theorem 1]{Kimeldorf-1978}.
%or \citet[Theorem 5.2.10]{fdsempi2016}.
In this situation, \((Y_n,X_n)\) converges weakly to \((Y,X)\); however, it holds that \(\xi(Y_n,X_n)=1\ne 0= \xi(Y,X)\) for all \(n\in \N\). Thus $\xi$ fails to be continuous with respect to weak convergence\footnote{This example on the lack of weak continuity of \(\xi\) is due to \citet[Remark 4.2(b)]{ansari2023MFOCI} in the second arXiv version of the paper. Later, a similar example has been given in \citet[Corollary 1.1]{buecher2024}.}.
\end{example}

The lack of weak continuity as illustrated in Example \ref{Cont.ExWC} has several important consequences. Most notably, the empirical distribution function cannot serve as a basis for a consistent plug-in estimator for \(\xi\). Further, the values of \(\xi\) are intrinsically hard to interpret in the sense that they can strongly deviate for distributions that are arbitrary close in any distance that metrizes weak convergence.
Another drawback, as shown by \cite{buecher2024}, is that independence tests based on \(\xi\) may exhibit only trivial power against certain alternatives that converge weakly to independence. Consequently, no meaningful uniform confidence intervals for \(\xi\) can exist.

To better understand the behavior of \(\xi\) across various models and to ensure robustness in applications such as variable selection or clustering (see \cite{chatterjee2021,deb2020b,sfx2024cluster}), a certain notion of continuity is essential.
Motivated by the representation of \(\xi\) through a conditionally independent copy in Lemma \ref{DimRed.Prop} below, we study in Section \ref{secxi} sufficient conditions for continuity of \(\xi\). 
%As we study in Section \ref{secxi}, Chatterjee's rank correlation \(\xi\) is continuous under weak convergence of Markov products. 
More precisely, we show in Theorem \ref{ThmCont} that the numerator of \(\xi\) in \eqref{defxi} is weakly continuous in the Markov product \((Y,Y'),\) where \(Y'\) is a conditionally independent copy of \(Y\) given \(\XX.\) 
%Since \(\XX\) and \(Y\) are independent if and only if $Y$ and $Y'$ are independent, our results suggest testing on independence between \(Y\) and \(Y'\) and not between \(X\) and \(Y\) to address the drawbacks identified in \cite{buecher2024}. \SF{Den letzten Satz streichen und die beiden Absätze zusammenfügen.}
Since the denominator in \eqref{defxi} depends on \(\overline{\Ran(F_Y)}\) (the closure of the range of \(F_Y\)), general continuity results for \(\xi\) also require convergence of the range of \(F_{Y_n}.\)
While range convergence is a rather simple concept, we provide in Section \ref{Sec.Cont.} various sufficient conditions for weak continuity of the Markov product \((Y,Y').\) Thereby, Theorem \ref{Cont.MainThm} on conditional weak continuity of \((Y,\XX)\) based on \cite{Sweeting_1989} serves as a main result of our paper. 
Further, Theorem \ref{thmdel} establishes copula-based continuity conditions for \(\xi\) using a generalized Markov product for copulas.
In many applications where the distributional assumptions hold only approximately, our results imply stability of \(\xi\) under slight deviations from the model assumptions.
In particular, we confirm in Proposition \ref{propadderrmod} robustness of \(\xi\) with respect to the parameter \(\sigma\) in the additive error model \eqref{eqadderrmod}. 
In Section \ref{secreldepm}, we study continuity of various dependence measures that are related to Chatterjee's rank correlation, before we show in Section \ref{Sec.Cont.PM} that all these measures are continuous in  standard models such as multivariate elliptical and \(\ell_1\)-norm symmetric distributions, or various copula models.
In particular, we obtain under some regularity conditions that Chatterjee's rank correlation is weakly continuous in these classes of models.
%For a better readability of the paper, several proofs are deferred to the appendix.

%%%%%%%%%%%%%%%%%%%%%%%%%%%%%%%%%%%%%%%%%%%%%%%%%%%%%%%%%%%%%%%%%%%%%%%%%%%%%%%%%%%%%%
%\bigskip\clearpage

\section{Continuity of \(\xi\)}\label{secxi}

For proving a first continuity result for $\xi$, stated in Theorem \ref{ThmCont} below, we shall use a representation of $\xi(Y,\XX)$ in terms of a dimension reduction principle that preserves the key information about the degree of functional dependence of $Y$ on $\XX$. Therefore, let \((Y,\XX)\) be a \((1+p)\)-dimensional random vector for some \(p\in \N.\)
We denote by \(Y'\) a conditionally independent copy of \(Y\) relative to \(\XX,\) i.e.,
\begin{align}\label{Assumption.DimR}
  (Y'|\XX = \xx ) \eqd (Y|\XX = \xx) 
  \textrm{ for } P^\XX\text{-almost all } \xx\in \R^p \text{ and } Y \perp_\XX Y'  ,
\end{align}
where \(\eqd\) indicates equality in distribution and \(Y \perp_\XX Y'\) denotes conditional independence of \(Y\) and \(Y'\) given \(\XX.\)
We refer to the distribution of the bivariate random vector \((Y,Y')\) in \eqref{Assumption.DimR} as the \emph{Markov product} of \((Y,\XX)\) and point to related concepts such as the Markov product for copulas studied by \cite{darsow-1992, lageras-2010} and the Pick-Freeze sampling scheme by \cite{Gamboa-Klein-Lagnoux-2018, janon2014}.
A basic property of the Markov product is that
\(Y\) and \(\XX\) are independent if and only if $Y$ and $Y'$ are independent. 
Further, $Y$ perfectly depends on $\XX$ if and only if $Y = Y'$ almost \textit{surely}; see \cite{fuchs2025ArxivMP}.

We make use of the following representation which states that \(\xi(Y,\XX)\) only depends on the diagonal of the distribution function of the Markov product \((Y,Y')\)
%\(\psi_{Y|\XX}\) 
and on the closure of the range of \(F_Y\). Recall that \(\Ran(F_Y)\) is the range of the distribution function of \(Y,\) and \(\overline{A}\) denotes the closure of a set \(A\subseteq [0,1].\)

\begin{lemma}[Representation of $\xi$; {\citet[Proposition 2.5]{ansari2023MFOCI}}]~~\label{DimRed.Prop}\\
For \((Y,Y')\) in \eqref{Assumption.DimR}, Chatterjee's rank correlation satisfies
\begin{align}\label{lemrepT1}
  \xi(Y,\XX) 
  & = a \int_{\R} P(Y < y, Y' < y) \de P^Y(y) - b
\end{align}
for positive constants 
\(a := (\int_{\R} \Var(\1_{\{Y\geq y\}})\de P^Y(y))^{-1}\) and 
\(b := a \int_{\R} P(Y<y)^2 \de P^Y(y),\) both depending only on \(\overline{\Ran(F_Y)}.\)
\end{lemma}
If \(F_Y\) is continuous, then \(\overline{\Ran(F_Y)} = [0,1],\) and the constants in \eqref{lemrepT1} take the values \(a = 6\) and \(b = 2\); see \citet[Section 3.1]{Gamboa-Klein-Lagnoux-2018} when \(F_Y\) admits a Lebesgue-density.
Further, note that for distribution functions \(F\) and \(G,\) we have the identity 
\begin{align}
    \overline{\Ran(F)} = \overline{\Ran(G)} \quad \Longleftrightarrow \quad F\circ F^{-1}(t) = G\circ G^{-1}(t) \quad \text{for } \lambda\text{-almost all } t \in (0,1),
\end{align}
where \(\lambda\) denotes the Lebesgue measure; see \citet[Proposition 2.14]{Ansari-2021}.
Hence, to obtain convergence of \(\xi(Y_n,\XX_n),\) the representation in \eqref{lemrepT1} motivates to study convergence of \(\int_{\R} P(Y_n < y, Y_n' < y) \de P^{Y_n}(y)\) and \(F_{Y_n}\circ F_{Y_n}^{-1}.\)  %This is the content of the following result

%\JA{Bezeichnung einführen für die Konvergenzen im folgenden Theorem?  \(\conv\)-product convergence and range convergence? Offensiv in einer Definition oder defensiv in der nachfolgenden Bemerkung?
%conditional weak convergence von \((\XX_n,Y_n)\) ist hinreichend für \(\conv\)-product convergence von \((Y_n,Y_n')\), aber nicht notwendig, da \(\XX_n\) bijektiv transformiert werden kann.} 

The following theorem gives sufficient conditions for continuity of \(\xi\) in terms of weak convergence of Markov products and in terms of range convergence of the marginal distribution functions. We write \(\VV_n\xrightarrow{~d~} \VV\) for a sequence of random vectors \((\VV_n)_{n\in \N}\) converging in distribution to a random vector \(\VV.\)

\begin{theorem}[Continuity of \(\xi\) based on Markov products]~~\label{ThmCont}\\
For random vectors \((Y_n,\XX_n)\) and \((Y,\XX)\), let \(Y'\) and, similarly, \(Y_n'\) be defined as in \eqref{Assumption.DimR}. If
\begin{enumerate}[(i)]
    \item \label{ThmCont1} \((Y_n,Y_n') \xrightarrow{~d~} (Y,Y')\) and
    \item \label{ThmCont2} \((F_{Y_n}\circ F_{Y_n}^{-1})(t)\to (F_Y\circ F_Y^{-1})(t)\) for \(\lambda\)-almost all \(t\in (0,1),\)
\end{enumerate}
 then \(\lim_{n \to \infty} \xi(Y_n,\XX_n) = \xi(Y,\XX).\) 
\end{theorem}
 Note that in the above theorem \(\XX_n\) and \(\XX\) can have different dimensions.

\begin{remark} \label{RemCont}
\begin{enumerate}[(a)]
\item Condition \eqref{ThmCont1} of Theorem \ref{ThmCont} states that the sequence of Markov products \((Y_n,Y_n')\) converges in distribution to \((Y,Y').\) It guarantees continuity of the numerator in \eqref{defxi}. By the continuous mapping theorem, dependence measures that can be represented as continuous functionals of \((Y,Y')\) are continuous with respect to weak convergence of the Markov products; see Section \ref{secmeaexpl} for continuity results on a measure of explainability. 
A similar continuity condition as in \eqref{ThmCont1} is given in \citet[Proposition 8]{deb2020b} for the kernel partial correlation coefficient.
In Section \ref{sec2}, we provide several sufficient conditions for weak convergence of \((Y_n,Y_n')\) to \((Y,Y').\) 
%In particular, we show in Theorem \ref{Cont.MainThm} that conditional weak convergence of \(Y_n|\XX_n = \xx\) as studied in \cite{Sweeting_1989} implies the weak convergence of \((Y_n,Y_n').\) 
In particular, we show by Theorem \ref{Cont.MainThm} that, under an additional continuity assumption on the conditional distributions \(Y_n\mid \XX_n = \xx,\) weak convergence of \((Y_n,\XX_n)\) to \((Y,\XX)\) implies weak convergence of the associated Markov products.
%The concept of conditional weak convergence is stronger than weak convergence. However, not even weak convergence of \((\XX_n,Y_n)\) to \((Y,\XX)\) is necessary for the weak convergence in \eqref{ThmCont} because the latter convergence is not affected by bijective transformations of \(\XX_n.\)
%\item 
Note that condition \eqref{ThmCont1} is neither sufficient nor necessary for weak convergence of the random vectors \((Y_n,\XX_n)\) to \((Y,\XX)\); see Examples \ref{ex2nsuf} and \ref{ex2nnec} below.
 %Due to Example \ref{ex2nnec} below, weak convergence of bivariate random vectors \((Y_n,X_n)\) to \((Y,X)\) does not imply \((Y_n,Y_n')\xrightarrow{~d~} (Y,Y').\)
Consequently, Chatterjee's rank correlation underlies a different mode of convergence compared to measures of concordance such as Kendall's tau or Spearman's rho. 
\item \label{RemCont:2}
Condition \eqref{ThmCont2} of Theorem \ref{ThmCont} states that the ranges of the distribution functions \(F_{Y_n}\) converge for Lebesgue-almost all points to the range of the limiting distribution function \(F_Y.\) 
%cf.~\citet[Proposition 2.14]{Ansari-2021}. 
It guarantees that the denominator in \eqref{defxi} converges. Note that this condition is neither necessary nor sufficient for weak convergences of \((Y_n)_{n \in \mathbb{N}}\) to \(Y\); see \citet[Examples 2.19 and 2.20]{Ansari-2021}. 
Typical approximations of distribution functions satisfy the concept of range convergence; see \citet[Examples 2.18]{Ansari-2021}.
If all \(Y_n\) and \(Y\) have a continuous distribution function, then condition \eqref{ThmCont2} is trivially fulfilled.
%Similar continuity conditions as in Theorem \ref{ThmCont} are given in \citet[Proposition 8]{deb2020b} for the kernel partial correlation coefficient. In Theorem \ref{Cont.MainThm}, we establish general conditions on \((\XX_n,Y_n)\) to imply weak convergence of the associated Markov products. More precisely, under an additional continuity assumption on the conditional distributions \(Y_n\mid \XX_n = \xx,\) weak convergence of \((\XX_n,Y)\) to \((Y,\XX)\) implies weak convergence of the associated Markov products. \SF{Die zweite Hälfte steht schon in (i).}
\end{enumerate}
\end{remark}

\begin{proof}[Proof of Theorem \ref{ThmCont}]
Denote by \(F_n\) and \(F\) the distribution function of \(Y_n\) and \(Y\), respectively. 
Further, denote by \(F_n^-\) and \(F^-\) their left-continuous versions, i.e., \(F_n^-(x):=\lim_{y\uparrow x} F_n(y)\) and \(F^-(x):=\lim_{y\uparrow x} F(y)\) for all \(x\in \R.\) 
Since, by assumption \eqref{ThmCont2}, \((F_n\circ F_n^{-1})(t) \to (F\circ F^{-1})(t)\) for \(\lambda\)-almost all \(t\), we also obtain for the left-continuous versions that
\begin{align}\label{lcvvs}
    (F_n^-\circ F_n^{-1})(t) \to (F^-\circ F^{-1})(t) ~~~ \text{for } \lambda\text{-almost all } t\in (0,1);
\end{align}
see \citet[Lemma 2.17.(ii)]{Ansari-2021}.
For the positive constants \(a,b\) in Lemma \ref{DimRed.Prop} and for their counterparts \(a_n,b_n\) related to $\xi(Y_n,\XX_n)$, we obtain
\begin{align}\label{convgjg1}
    a_n \to a~~~ \text{and}~~~ b_n \to b \quad \text{as } n\to \infty
\end{align}
because
\begin{align*}
    a_n^{-1} 
    &= \int_\R \var\left(\1_{\{Y_n\geq y\}}\right) \de P^{Y_n}(y) 
     = \int_{(0,1)} \var\left(\1_{\{Y_n\geq F_n^{-1}(t)\}}\right) \de t \\ 
    &= \int_{(0,1)} \left[P\left(Y_n\geq F_n^{-1}(t)\right) - P\left(Y_n\geq F_n^{-1}(t)\right)^2\right] \de t\\
    &= \int_{(0,1)}\left[1-F_n^-\circ F_n^{-1}(t)\right]
       - [1-F_n^-\circ F_n^{-1}(t)]^2 \de t \\
    & \to \int_{(0,1)}\left[1-F^-\circ F^{-1}(t)\right]
       - [1-F^-\circ F^{-1}(t)]^2 \de t = \cdots = \int_\R \var\left(\1_{\{Y\geq y\}}\right) \de P^{Y}(y) = a^{-1}
\end{align*}
and
\begin{align*}
    \int_\R P(Y_n < y)^2 \de P^{Y_n}(y) 
    & = \int_{(0,1)} (F_n^-\circ F_n^{-1}(t))^2 \de t \\
    & \xrightarrow[n\to \infty]{}\int_{(0,1)} (F^-\circ F^{-1}(t))^2 \de t = \int_\R \lim_{z\uparrow y} P(Y < z)^2 \de P^{Y}(y),
\end{align*}
due to \eqref{lcvvs} and dominated convergence.
It remains to prove 
\begin{align}\label{convgjg2}
    \lim_{n\to \infty} \int_\R  P(Y_n < y, Y_n' < y) \de P^{Y_n}(y) 
    = \int_\R P(Y < y, Y' < y) \de P^{Y}(y).
\end{align}
Then the statement follows from
\begin{align*}
    \lim_{n\to \infty} \xi(Y_n,\XX_n) 
    &= \lim_{n\to \infty} \left(a_n \int_\R  P(Y_n < y, Y_n' < y) \de P^{Y_n}(y) - b_n\right) \\
     &= a \int_\R P(Y < y, Y' < y) \de P^{Y}(y) -b 
     = \xi(Y,\XX)
\end{align*}
applying \eqref{convgjg1} and \eqref{convgjg2} as well as the representation of \(\xi\) due to Lemma \ref{DimRed.Prop}.
\\
In order to prove \eqref{convgjg2}, we make use of the convergence \((Y_n,Y_n') \xrightarrow{~d~} (Y,Y')\) due to assumption \eqref{ThmCont1}.
To this end, we use Sklar's theorem (see e.g.~\cite{fdsempi2016} or \cite{Nelsen-2006}) to decompose the bivariate distribution function \(F_{(Y,Y')}\) into its univariate marginal distribution functions \(F_Y = F\) and \(F_{Y'} = F\) and a bivariate copula \(C\) such that
\begin{align}\label{eqsklar}
    F_{(Y,Y')}(y,y') = C(F(y),F(y')) \quad \text{for all } y,y'\in \R,
\end{align}
noting that a bivariate copula is a distribution function on \([0,1]^2\) with marginals that are uniform on \([0,1].\) The copula in \eqref{eqsklar} is uniquely determined on \(\Ran(F)\times \Ran(F),\) and, due to continuity of copulas, also on \(\overline{\Ran(F)}\times \overline{\Ran(F)}.\) By a similar decomposition for \((Y_n,Y_n')\) and by Portmanteau's lemma (see, e.g. \citet[Lemma 2.2]{vanderVaart-1998}), weak continuity of the Markov products means
\begin{align}\label{convcophgi}
    C_n(F_n(y),F_n(y')) 
    = F_{(Y_n,Y_n')}(y,y') 
    \xrightarrow[n\to \infty]{} F_{(Y,Y')}(y,y') 
    = C(F(y),F(y'))
\end{align}
for all continuity points \(y,y'\in \R\,\) of $F$, 
where \(C_n\) (and, similarly, for \(C\)) denotes the standard extension of the copula associated with \((Y_n,Y_n')\) from \(\overline{\Ran(F_n)}\times \overline{\Ran(F_n)}\) to \([0,1]^2\); see \cite{Neslehova2007}. 
For \((v,v')\) in the interior of \(\overline{\Ran(F)}\times \overline{\Ran(F)}\), \(y = F^{-1}(v)\) and \(y'=F^{-1}(v')\) are continuity points of \(F\) and
\begin{align}\label{convcophgi5}
    |C_n&(v,v') - C(v,v')| = |C_n(F(y),F(y')) - C(F(y),F(y'))| \\
    \label{convcophgi6} &\leq |C_n(F(y),F(y')) - C_n(F_n(y),F_n(y'))| + | C_n(F_n(y),F_n(y')) - C(F(y),F(y'))|.
\end{align}
The right-hand term in \eqref{convcophgi6} converges to \(0\) due to \eqref{convcophgi}. By Lipschitz-continuity of copulas (see, e.g.~\citet[Theorem 2.2.4]{Nelsen-2006}), the left-hand term in \eqref{convcophgi6} is upper bounded by \(|F(y)-F_n(y)| + |F(y')-F_n(y')|\), which converges to \(0\) because \(Y_n\xrightarrow[]{d} Y\).
%Since copulas are Lipschitz-continuous functions (see, e.g.~\citet[Theorem 2.2.4]{Nelsen-2006}), 
If \(v\) is an isolated point of \(\overline{\Ran(F)}\) and \(v'\) an inner point, choose \(y'\) as above and \(y = F^{-1}(v) + \varepsilon\) for \(\varepsilon >0\) small enough such that \(F(y) = v\). Then \(y\) is a continuity point of \(F\) and \(C_n(v,v')\) converges to \(C(v,v')\) with the same reasoning as in \eqref{convcophgi6}. The other cases follow similarly. Using Lipschitz-continuity of copulas, the convergence in \eqref{convcophgi5} to \(0\) holds true for all \(v,v'\in \overline{\Ran(F)}\). This implies
%noting that boundary points of \(\overline{\Ran(F)}\) can be handled with the first case using Lipschitz-continuity of copulas. In total, the convergence in \eqref{convcophgi5} to \(0\) for all \(v,v'\in \overline{\Ran(F)}\) implies 
\begin{align}\label{equnifconvuv}
    C_n \to C  ~~~\text{uniformly on } \overline{\Ran(F)}\times \overline{\Ran(F)},
\end{align}
using a variant of Arzel\`{a}-Ascoli's theorem due to \citet[Theorem 11.28]{Rudin-1987}.
%Since \(F_{(Y_n,Y_n')} (y,y) = C_n(F_n(y),F_n(y))\) and \(F_{(Y,Y')} (y,y) = C(F(y),F(y))\) for all \(y\in \R,\) we obtain for the left-continuous versions that
Consequently,
\begin{align*}
    P(Y_n < F_n^{-1}(t), Y_n' < F_n^{-1}(t)) 
    &= C_n(F_n^-\circ F_n^{-1}(t),F_n^-\circ F_n^{-1}(t))
    \\
    &\xrightarrow[n\to \infty]{} C(F^-\circ F^{-1}(t),F^-\circ F^{-1}(t)) 
     = P(Y < F^{-1}(t), Y' < F^{-1}(t)) 
\end{align*}
for \(\lambda\)-almost all \(t\in (0,1),\) where we use \eqref{equnifconvuv} and \eqref{lcvvs} as well as continuity of copulas for the convergence.
Finally, we obtain 
\begin{align*}
    \lim_{n\to \infty} \int_\R  & P(Y_n < y, Y_n' < y) \de P^{Y_n}(y) 
    = \int_{(0,1)} \lim_{n\to \infty} P(Y_n < F_n^{-1}(t), Y_n' < F_n^{-1}(t)) \de t\\
    &=  \int_{(0,1)} P(Y < F^{-1}(t), Y' < F^{-1}(t))  \de t 
    =  \int_\R P(Y < y, Y' < y) \de P^{Y}(y)
\end{align*}
applying dominated convergence. This proves \eqref{convgjg2}.
\end{proof}

In the following two examples, we show for bivariate random vectors that weak convergence of \((Y_n,X_n)\) is neither sufficient nor necessary for weak convergence of the associated Markov products.

\begin{example}[\((Y_n,X_n)\xrightarrow{~d~}(Y,X)\) does not imply \((Y_n,Y_n')\xrightarrow{~d~} (Y,Y')\)]\label{ex2nsuf}~\\
Consider the setting of Example \ref{Cont.ExWC}. Then, we have \((Y_n,X_n) \xrightarrow{~d~} (Y,X)\). However, \(Y\) and \(Y'\) are independent while, for each $n \in \mathbb{N}$, \(Y_n = Y_n'\) almost surely. Since \(Y_n,Y\) are uniform on \((0,1),\) we obtain that weak convergence of \((Y_n,X_n)\) to \((Y,X)\) is not sufficient for condition \eqref{ThmCont1} in Theorem \ref{ThmCont}; see also \citet[Theorem 5.2.10]{fdsempi2016}.
\end{example}

\begin{example}[\((Y_n,Y_n')\xrightarrow{~d~} (Y,Y')\) does not imply \((Y_n,X_n)\xrightarrow{~d~}(Y,X)\)]\label{ex2nnec}~\\
Let \(X\) be a non-degenerate random variable having a distribution that is symmetric around \(0\), for example \(X\) is standard normal. Consider \(Y =X\), \(X_n = X\), and \(Y_n = -X\) for all $n \in \mathbb{N}$.
Then, for all \(n\in \N,\) we have \((Y_n,X_n) = (-X,X) \not\eqd (X,X) = (Y,X)\),
but \(Y_n = Y_n'\) and \(Y = Y'\) almost surely by construction. Consequently, we have \((Y_n,Y_n') \eqd (Y,Y)\) due to symmetry of \(X.\) Hence, condition \eqref{ThmCont1} in Theorem \ref{ThmCont} is not sufficient for weak convergence of \((Y_n,X_n)\) to \((Y,X).\)
\end{example}

\section{Sufficient conditions for weak convergence of Markov products}\label{sec2}
\label{Sec.Cont.}

In this section, we provide various sufficient conditions for weak convergence of Markov products that yield continuity of \(\xi\) due to Theorem \ref{ThmCont}. The results in this section are also useful to obtain continuity of related dependence measures, as we show in Section \ref{secreldepm}. 
We begin with a main result of this paper that establishes convergence of the Markov products \((Y_n,Y_n')\) based on conditional weak convergence of \((Y_n,\XX_n).\) Then we show robustness of Markov products when incorporating noise into models. In the last part of this section, we provide a copula-based version of Theorem \ref{ThmCont} that yields continuity of \(\xi\) in the parameter of various copula families.

\subsection{Conditional weak convergence}

Chatterjee's rank correlation and related dependence measures are defined by comparing conditional distributions. Hence, studying continuity of these measures requires to study convergence of conditional distributions. To this end, we make use of the notion of conditional weak convergence for which we use the following concepts.

For \(d\in \N,\) denote by \(\cU(\R^d)\) a class of bounded, continuous, weak convergence-determining\footnote{A class \(\cU(\R^d)\) of measurable functions mapping from \(\R^d\) to \(\C\) is \emph{weak-convergence determining} if, for distributions \(\mu_n\), \(\mu\) on \(\R^d\), \(\int f \de\mu_n \to \int f \de \mu\) for all \(f\in \cU(\R^d)\) implies \(\mu_n\to \mu\) weakly.} functions mapping from \(\R^d\) to \(\C.\) 
%Denote by \(\xrightarrow{ d }\) convergence in distribution. 
A sequence \((f_n)_{n\in \N}\) of functions mapping from \(\R^d\) to \(\C\) is said to be \emph{asymptotically equicontinuous} on an open set \(V\subset \R^d,\) if for all \(\varepsilon>0\) and \(\xx \in V\) there exist \(\delta(\xx,\varepsilon)>0\) and \(n(\xx,\varepsilon)\in \N\) such that whenever \(|\xx'-\xx|\leq \delta(\xx,\varepsilon)\) then \(|f_n(\xx')-f_n(\xx)|<\varepsilon\) for all \(n > n(\xx,\varepsilon).\) Further, \((f_n)_{n\in \N}\) is said to be \emph{asymptotically uniformly equicontinuous} on \(V\) if it is asymptotically equicontinuous on \(V\) and the constants \(\delta(\varepsilon)=\delta(\xx,\varepsilon)\) and \(n(\varepsilon)=n(\xx,\varepsilon)\) do not depend on \(\xx.\)

The following main result is based on a characterization of conditional weak convergence by \cite{Sweeting_1989} and gives general sufficient conditions for weak convergence of Markov products.
Here, we assume that \(\XX_n\) and \(\XX\) have the same dimension.

\begin{theorem}[Conditional weak convergence] \label{Cont.MainThm}\label{Cont.MainLemma}~\\
Consider the \((1+p)\)-dimensional random vector $(Y,\XX)$ and a sequence $(Y_n,\XX_n)_{n\in \N}$ of \((1+p)\)-dimensional random vectors.
Let \(V\subset \R^p\) be open such that \(P(\XX\in V) =1\).
For $(Y_n')_{n \in \mathbb{N}}$ and $Y'$ defined as in equation \eqref{Assumption.DimR},
the following statements \eqref{Cont.MainThm1} and \eqref{Cont.MainThm2} are equivalent:
\begin{enumerate}[(1)]
\item \label{Cont.MainThm1}
  \begin{enumerate}[(i)]
  \item \label{Cont.MainLemma.1}\label{Cont.MainThm.1} \((Y_n,\XX_n)\xrightarrow{~d~} (Y,\XX),\) and
  \item \label{Cont.MainLemma.2}\label{Cont.MainThm.2} \((\E[u(Y_n)|\XX_n=\,\cdot\,])_{n\in \N}\) is asymptotically equicontinuous on \(V\) for all \(u\in \cU(\R),\) 
  \end{enumerate}
\item \label{Cont.MainThm2}
  \begin{enumerate}[(i)]
  \item \label{Cont.MainLemma.3} \((Y_n,Y_n',\XX_n) \xrightarrow{~d~} (Y,Y',\XX)\), and 
  \item \label{Cont.MainLemma.4} \((\E[u(Y_n,Y_n')|\XX_n=\,\cdot\,])_{n\in \N}\) is asymptotically equicontinuous on \(V\) for all \(u\in \cU(\R^2)\). 
  \end{enumerate}
\end{enumerate}
\end{theorem}

\begin{remark}\label{Cont.MainRem}
%\begin{enumerate}[(i)]
%\item 
Theorem \ref{Cont.MainThm} states that, under assumption \eqref{Cont.MainThm.2} on asymptotic equicontinuity of the conditional distributions, weak convergence of \((Y_n,\XX_n)\) implies weak convergence of the Markov products \((Y_n,Y_n').\) Hence, if also the ranges of \(F_{Y_n}\) converge, we obtain from Theorem \ref{ThmCont} that \(\xi(Y_n,\XX_n)\) converges. These conditions can be verified for many models as we show in Section \ref{Sec.Cont.PM.Ellipt.} within the class of elliptical distributions.
%\item 
%\end{enumerate}
\end{remark}

\begin{proof}[Proof of Theorem \ref{Cont.MainThm}.]
For an arbitrary sequence \((g_n)_{n\in \N}\) of complex functions on a metric space \(S,\) denote by \(g_n\xrightarrow{~u~} g,\) \(g\) a continuous function, the convergence \(g_n(s)\to g(s)\) uniformly on compact subsets of \(S.\)
Further, denote by \(\cC_b(\R^d)\) the class of all bounded continuous functions on \(\R^d\).

Assume \eqref{Cont.MainThm1}.
Due to the characterization of uniform conditional convergence in \citet[Theorem 4]{Sweeting_1989}, conditions \eqref{Cont.MainLemma.1} and \eqref{Cont.MainLemma.2} are equivalent to
\begin{align}\label{eqlemconvT1}
    \int_\R f(y) \de P^{Y_n|\XX_n=\xx}(y) &\xrightarrow{~u~} \int_\R f(y) \de P^{Y|\XX=\xx}(y) ~~~\text{on } V \text{ for all }f \in \cC_b(\R),~\text{and}\\
    \label{eqlemconvT2} \XX_n &\xrightarrow{~d~} \XX.
\end{align}
This implies for the characteristic functions of the conditional distribution \((Y_n,Y_n')|\XX_n=\xx\) that
\begin{align}\label{eqlemconvT19}
\begin{split}
    & \left| \varphi_{(Y_n,Y_n')|\XX_n=\xx}(t,t')-\varphi_{(Y,Y')|\XX=\xx}(t,t')\right|
    = \left|\varphi_{Y_n|\XX_n=\xx}(t) \, \varphi_{Y_n'|\XX_n=\xx}(t')-\varphi_{Y|\XX=\xx}(t) \, \varphi_{Y'|\XX=\xx}(t')\right|
    \\
    &= \left|\varphi_{Y_n|\XX_n=\xx}(t) \, \varphi_{Y_n|\XX_n=\xx}(t')-\varphi_{Y|\XX=\xx}(t)\varphi_{Y|\XX=\xx}(t')\right|
    \\
    &\leq \left| \varphi_{Y_n|\XX_n=\xx}(t)-\varphi_{Y|\XX=\xx}(t)\right|\left|\varphi_{Y_n|\XX_n=\xx}(t')\right|+ \left| \varphi_{Y_n|\XX_n=\xx}(t')-\varphi_{Y|\XX=\xx}(t')\right| 
    \left|\varphi_{Y|\XX=\xx}(t)\right|
    \\
    &\leq \left| \varphi_{Y_n|\XX_n=\xx}(t)-\varphi_{Y|\XX=\xx}(t)\right| + \left| \varphi_{Y_n|\XX_n=\xx}(t')-\varphi_{Y|\XX=\xx}(t')\right|
    %\nonumber 
    %\label{Uconvs}&\xrightarrow[]{~u~} 0
\end{split}
\end{align}
for all \(\xx \in V\) and for all \(t,t'\in \R\), where we use for the equality that \(Y_n'\) is a conditionally independent copy of \(Y_n\) given \(\XX_n\), and, similarly, for \(Y'\) and \(Y\) under \(\XX\).  The last inequality holds true since absolute values of characteristic functions are upper bounded by \(1\).
Since \(\varphi_{Y_n|\XX_n=\xx} \xrightarrow[]{u} \varphi_{Y|\XX=\xx}\) by 
 \eqref{eqlemconvT1}, the limiting conditional characteristic function \(\varphi_{(Y,Y')|\XX=\xx}\) is continuous in the conditioning variable, i.e.,
\begin{align}\label{setur7}
    \xx_n\to \xx ~~~ \Longrightarrow ~~~ \varphi_{(Y,Y')|\XX=\xx_n}(t,t') \xrightarrow[n\to \infty]{} \varphi_{(Y,Y')|\XX=\xx}(t,t') ~~~\text{for all } t,t'\in \R.
\end{align}
Further, %by \eqref{eqlemconvT19},
\((\varphi_{(Y_n,Y_n')|\XX_n=\cdot}(t,t'))_n\) continuously converges to \(\varphi_{(Y,Y')|\XX=\cdot}(t,t')\) for all \(t,t' \in \R,\) i.e.,
\begin{align}\label{setur6}
    \xx_n\to \xx ~~~ \Longrightarrow ~~~ \varphi_{(Y_n,Y_n')|\XX_n=\xx_n}(t,t')\xrightarrow[n\to \infty]{} \varphi_{(Y,Y')|\XX=\xx}(t,t')~~~\text{for all } t,t'\in \R.
\end{align}
%Due to \eqref{Uconvs}, we obtain
%\begin{align}\label{setur1}
%   \sup_{\xx\in I}\left|\varphi_{(Y_n,Y_n')|\XX_n=\xx}(t,t')-\varphi_{(Y,Y')|\XX=\xx}(t,t')\right| \longrightarrow 0
%\end{align}
%for all compact sets \(I\subset V.\) 
%Since, by assumption, \((\varphi_{Y_n|\XX_n=\xx}(t))_n\) is asymptotically equicontinuous on \(V\) for all \(t\) 
%%and since by \eqref{eqlemconvT1} \(\varphi_{Y_n|\XX_n=\xx}(t)\xrightarrow{~u~} \varphi_{Y|\XX=\xx}(t),\) we obtain with \eqref{Uconvs} that 
%Now, observe that \citet[Condition (2.5)]{Sethuraman-1961} coincides with \eqref{setur6} and \citet[Condition (2.7)]{Sethuraman-1961} coincides with \eqref{setur7}. 
%Noting that  \citet[Lemma 3 \& Lemma 4]{Sethuraman-1961} can alternatively be proven using \citet[Conditions (2.5) \& (2.7)]{Sethuraman-1961} \JA{Lemma 3 und 4 werden doch mit (2.5) und (2.7) bewiesen, oder?} instead of \citet[Conditions (2.2) \& (2.3)]{Sethuraman-1961}, 
%It follows that
Now, applying \citet[Lemma 3 \& Lemma 4]{Sethuraman-1961}, we obtain
\begin{align}\label{eqlemconvT1.2}
    \int_{\R^2} f(y,y') \de P^{(Y_n,Y_n')|\XX_n=\xx}(y,y') 
    &\xrightarrow{~u~} \int_{\R^2} f(y,y') \de P^{(Y,Y')|\XX=\xx}(y,y') ~ \text{on } V \text{ for all }f \in \cC_b(\R^2).
\end{align}
Note that \citet[Lemma 3 \& Lemma 4]{Sethuraman-1961} are proven using \citet[(2.5) \& (2.7)]{Sethuraman-1961} where \citet[(2.5)]{Sethuraman-1961} coincides with \eqref{setur6} and \citet[(2.7)]{Sethuraman-1961} coincides with \eqref{setur7}.
Using again the characterization of uniform conditional convergence in \citet[Theorem 4]{Sweeting_1989} together with \eqref{eqlemconvT2} finally yields conditions \eqref{Cont.MainLemma.3} and \eqref{Cont.MainLemma.4}. 
The reverse direction is trivial.
\end{proof}

\subsection{Noise resistance} \label{Subsec.Noise}

We now establish weak continuity of the Markov product and continuity of \(\xi\) in the noise of an additive error model that generalizes the setting in \eqref{eqadderrmod}.
%\JA{Verallgemeinern auf 'weakly continuous in the noise \(\varepsilon\)}
%associated with \SF{a generalization of the additive error model in \eqref{eqadderrmod}} and hence obtain, as a practical application of Theorem \ref{ThmCont}, continuity of \(\xi\) in the parameter \(\sigma.\)
%As a practical application of Theorem \ref{ThmCont},
%we now prove robustness of \(\xi\) in the parameter \(\sigma\) of the additive error model \eqref{eqadderrmod}. 

\begin{proposition}[Additive error model]\label{propadderrmod}~\\
Consider the model \(Y = f(\XX) + \varepsilon\). Assume that \(f(\XX)\) and the noise \(\varepsilon\) are independent and that \(f(\XX)\) has a continuous distribution function. 
 Then the Markov product \((Y,Y')\) of \((Y,f(\XX))\) is weakly continuous and \(\xi(Y,\XX)\) is continuous with respect to weak convergence in \(\varepsilon.\)
\end{proposition}

\begin{proof} 
    Let \(\varepsilon_n\) be a sequence of random variables with \(\varepsilon_n\xrightarrow{d} \varepsilon\) and let \(\varepsilon_n'\) and \(\varepsilon'\) be copies of \(\varepsilon_n\) resp. \(\varepsilon\) such that \(f(\XX),\varepsilon,\varepsilon',\varepsilon_n,\varepsilon_n'\) are independent. 
    Consider the sequences \((Y_n,X_n)_{n\in \N}\) and \((Y_n',X_n)_{n\in \N}\) of bivariate random vectors defined by \(Y_n = f(\XX) + \varepsilon_n,\) \(Y_n' = f(\XX) +\varepsilon_n',\) and \(X_n = f(\XX)\).
    Then \(Y_n\) and \(Y_n'\) are conditionally i.i.d given \(f(\XX).\) 
    By independence of \(f(\XX), \varepsilon_n,\varepsilon_n',\varepsilon,\varepsilon'\), we have
    \begin{align*}
        \E h(Y_n,Y_n') %= \int_\R \E[h(f(\XX)+\varepsilon,f(\XX) + \varepsilon_n')\mid f(\XX) = z] \de P^{f(\XX)}(z) 
        = \int_\R \E h(z+\varepsilon_n,z+\varepsilon_n') \de P^{f(\XX)}(z) \to \int_\R \E h(z+\varepsilon,z+\varepsilon') \de P^{f(\XX)}(z) = \E h(Y,Y')
    \end{align*}
    for all continuous and bounded functions \(h\). Hence, 
    \((Y_n,Y_n')\) converges weakly to \((Y,Y')\eqd(f(\XX)+\varepsilon, f(\XX)+\varepsilon')\), and thus \((Y,Y')\) is weakly continuous in \(\varepsilon.\) 
    Further, since \(f(\XX)\) has a continuous distribution function, all of \(Y_n,Y_n',Y,Y'\) have a continuous distribution function. Hence, \((F_{Y_n}\circ F_{Y_n}^{-1})(t) = t = (F_Y\circ F_Y^{-1})(t)\) for all \(n\) and \(t\in (0,1).\) Then, continuity of \(\xi(Y,f(\XX))\) in \(\varepsilon\) follows from Theorem \ref{ThmCont}.
    Finally, since \(Y\) and \(\XX\) are conditionally independent given \(f(\XX)\), self-equitability of \(\xi\) \citet[Corollary 2.3]{ansari2023MFOCI} implies \(\xi(Y,\XX) = \xi(Y,f(\XX))\) and hence continuity of \(\xi(Y,\XX)\) in \(\varepsilon\).
\end{proof}

The following result ensures that a certain level of noise present in the data does not cause the Markov product to deviate too much.

\begin{proposition}[Robustness against small perturbations of the response]~~\label{RobustY}\\
For a \((1+p)\)-dimensional random vector \(( Y,\XX)\) and a sequence \(( \varepsilon_n)_{n\in\mathbb{N}}\) of random variables, consider the perturbations \(Y_n := Y + \varepsilon_n.\) 
If
\begin{enumerate}[(i)]
\item \label{RobustY.A1} \((\varepsilon_n\mid \XX = \xx ) \xrightarrow{d} 0\) for \(P^\XX\)-almost all \(\xx\) and,
\item \label{RobustY.A2} for every $n \in \N$, $Y$ and $\varepsilon_n$ are conditionally independent given $\XX$,
%%% Revision#1
%\item\label{RobustY.A1} \(\varepsilon_n \xrightarrow{d} 0\) and,
%\item\label{RobustY.A2} for every $n \in \N$, $\varepsilon_n$ is independent from $(Y,\XX)$,
%%%
%the mapping \(V \to \C, \xx \mapsto \varphi_{Y|\XX=\xx}\) is continuous, where $\varphi_{Y|\XX=\xx}$ denotes the characteristic function of the conditional distribution $Y|\XX=\xx$,
\end{enumerate}
then \((Y_n,Y_n') \xrightarrow{d} (Y,Y')\) where $Y'$ and, similarly, $Y_n'$ are defined as in \eqref{Assumption.DimR}.
%\SF{If additionally, condition \eqref{ThmCont2} in Theorem \ref{ThmCont} is fulfilled, then \(\lim_{n \to \infty} \xi(Y_n,\XX_n) = \xi(Y,\XX).\)}
\end{proposition}
\begin{proof}
Due to \eqref{Assumption.DimR} and since $\varepsilon_n$ and $Y$ are conditionally independent given \(\XX\), we have for \(P^\XX\)-almost all  $\xx, \xx'$ and for all $(t,t') \in \R^2$ that
\begin{align*}%\label{RobustY.Eq1}
    \varphi_{(Y_n,Y_n') \mid \XX=\xx} (t,t')
    & = \varphi_{Y_n \mid \XX=\xx} (t) \, \varphi_{Y_n' \mid \XX=\xx} (t')
    = \varphi_{Y + \varepsilon_n \mid \XX=\xx} (t) \, \varphi_{Y+\varepsilon_n \mid \XX=\xx} (t') \notag
    \\
    & = \varphi_{Y \mid \XX=\xx} (t) \, \varphi_{\varepsilon_n \mid \XX=\xx} (t) \, \varphi_{Y \mid \XX=\xx} (t') \, \varphi_{\varepsilon_n \mid \XX=\xx} (t') \notag
    \\
    & = \varphi_{(Y,Y') \mid \XX=\xx} (t,t') \, \varphi_{\varepsilon_n \mid \XX=\xx} (t) \, \varphi_{\varepsilon_n \mid \XX=\xx} (t') \notag
    %\\
    %& = \varphi_{(Y,Y') \mid \XX=\xx} (t,t') \, \varphi_{\varepsilon_n} (t) \, \varphi_{\varepsilon_n} (t') 
    \xrightarrow[n\to \infty]{} \varphi_{(Y,Y') \mid \XX=\xx} (t,t'), \notag
\end{align*}
where 
convergence follows from \((\varepsilon_n\mid \XX = \xx ) \xrightarrow{d} 0\).
Integrating over $P^\XX$ yields \((Y_n,Y_n') \xrightarrow{d} (Y,Y')\).
\end{proof}
Proposition \ref{RobustY} together with Theorem \ref{ThmCont} guarantees robustness of \(\xi\) against slight misspecifications of the model. 
Concretely, when the assumed model is \(Y = f(\XX) + \varepsilon\) while the true data-generating process is \(Y = f(\XX) + \gamma \, g(\XX) + \varepsilon\) with \(\gamma\) close to \(0\), \(\xi\) under the simpler model will remain close to its true value.
%In order to meet \eqref{ThmCont2} of Theorem \ref{ThmCont}, we require that either \(f(\XX)\) or \(\varepsilon\) has a continuous distribution function.

\subsection{\(\partial_2\)-convergence of copulas}

In this subsection, we focus on weak continuity of Markov products based on convergence criteria for the underlying copulas. In particular, we obtain for various parametric copula families simple criteria to verify weak continuity of their Markov products in the underlying copula parameters. 
For simplicity, we study only the bivariate case. 
In analogy to the Markov product for random vectors defined at the beginning of Section \ref{secxi}, the Markov product of a bivariate copula \(D\) is defined by 
\begin{align}\label{defcopprod}
    D\conv D(u,v) := \int_0^1 \partial_2 D(u,t) \partial_2 D(v,t) \de t \quad \text{for } u,v\in [0,1],
\end{align}
where \(\partial_2\) denotes the partial derivative operator with respect to the second component. Note that, for fixed \(u\in [0,1],\) the partial derivative \(\partial_2 D(u,t)\) exists for \(\lambda\)-almost all \(t\in (0,1)\); see \citet[Theorem 2.2.7]{Nelsen-2006}.
If \(D\) is a copula of \((Y,X)\) and if \(F_X\) is continuous, then \(D\conv D\) is a copula of the Markov product \((Y,Y')\); see e.g. \citet[Section 3]{darsow-1992}.
%Recall that, by \eqref{eqsklar}, every bivariate distribution function can be decomposed into a concatenation of its marginal distribution functions and a bivariate copula.

As a consequence of Example \ref{ex2nsuf}, weak continuity of \(D\) is not sufficient for weak continuity of \(D\conv D.\) 
%Therefore, we need to use a stronger mode of convergence guaranteeing also convergence of the partial copula derivatives.
To ensure weak convergence of Markov products for copulas, we make use of the concept of \(\partial_2\)-convergence, defined for bivariate copulas \((D_n)_{n\in \N}\) and \(D\) by
\begin{align}
    D_n \xrightarrow{~\partial_2 ~} D \quad \Longleftrightarrow \quad \int_0^1 \left|\partial_2 D_n(v,t) - \partial_2 D(v,t) \right| \de t \to 0 \quad \text{for all } v\in [0,1];
\end{align}
see \citet[Definition 2.21 and Theorem 2.23]{Ansari-2021}. Note that \(\partial_2\)-convergence implies weak convergence of copulas, and it is a weaker concept than \(\partial\)-convergence considered by \cite{Mikusinski-2010}.

The following main result is a copula-based version of Theorem \ref{ThmCont} for bivariate random vectors. It states that range convergence of the marginals and \(\partial_2\)-convergence of the copulas of \((Y_n,X_n)\) to the copula of \((Y,X)\) implies convergence of Chatterjee's rank correlation. Note that we do not require weak convergence of \(Y_n\) to \(Y\) or \(X_n\) to \(X\).

\begin{theorem}[\(\partial_2\)-convergence and range convergence]\label{thmdel}~\\
    For \(n\in \N,\) let \((Y_n,X_n)\) and \((Y,X)\) be bivariate random vectors. Assume that
    \begin{enumerate}[(i)]
        \item \label{thmdel1} \(C_{Y_n,X_n} \xrightarrow{\partial_2} C_{Y,X},\)
        \item \label{thmdel2} \(F_{X_n} \circ F_{X_n}^{-1}(t) \to F_{X} \circ F_{X}^{-1}(t)\) for  \(\lambda\)-almost all \(t\in (0,1).\)
        \end{enumerate}
        Then we have \(C_{Y_n,Y_n'} \to C_{Y,Y'}\) uniformly on \(\overline{\Ran(F_Y)}\times \overline{\Ran(F_Y)}.\) If additionally
    \begin{enumerate}[(i)]\setcounter{enumi}{2}
        \item \label{thmdel3} \(F_{Y_n} \circ F_{Y_n}^{-1}(t) \to F_{Y} \circ F_{Y}^{-1}(t)\) for \(\lambda\)-almost all \(t\in (0,1),\)
    \end{enumerate}
        it follows that \(\xi(Y_n,X_n) \to \xi(Y,X).\)
\end{theorem}

\begin{proof}
%Denote by \(F^-\) the left-continuous version of a distribution function \(F.\)
    Define the rank transformed random variables \(Z_n := F_{Y_n}(Y_n)\) and \(Z := F_Y(Y).\) 
    Since Chatterjee's rank correlation is invariant under distributional transformations, we have \(\xi(Y_n,X_n) = \xi(Z_n,X_n)\) and \(\xi(Y,X) = \xi(Z,X)\); see \citet[Proposition 2.4]{ansari2023MFOCI}. %, which also applies to the left-continuous version of the associated distribution functions. 
    We show that
    \begin{align}
        \label{eqthmdel1}(Z_n,Z_n') &\xrightarrow{~d~} (Z,Z') \\
        \label{eqthmdel2}\text{and} \quad F_{Z_n}\circ F_{Z_n}^{-1}(t) &\xrightarrow{~\phantom{d}~} F_Z \circ F_Z^{-1}(t) \quad \text{for } \lambda\text{-almost all } t\in (0,1).
    \end{align}
    Then, the convergence of \(\xi(Y_n,X_n)\) to \(\xi(Y,X)\) follows from Theorem \ref{ThmCont}. The convergence in \eqref{eqthmdel2} is a consequence of assumption \eqref{thmdel3} using that \(\overline{\Ran(F_{Y_n})} = \overline{\Ran(F_{Z_n})}\) and  \(\overline{\Ran(F_Y)} = \overline{\Ran(F_Z)}.\) 
    For the proof of \eqref{eqthmdel1}, we first show weak convergence of the marginals. Therefore, define \(Z_n^* := (F_{Y_n}\circ F_{Y_n}^{-1})(V)\) and \(Z^* := (F_{Y}\circ F_{Y}^{-1})(V)\) for some random variable \(V\) that is uniformly distributed on \((0,1).\) Then \(Z_n^* \eqd Z_n\) and \(Z^*\eqd Z.\) 
    Using assumption \eqref{thmdel3}, \(Z_n^*\) converges to \(Z^*\) almost surely and, thus, we obtain \(Z_n\xrightarrow{~d~} Z.\)\\
    To prove convergence of the bivariate random vectors in \eqref{eqthmdel1}, 
    note that the copula of \((Y_n,Y_n')\) can be written as the conditional independence product of the copula \(C_{Y_n,X_n}\) with itself and with respect to the distribution function \(F_{X_n}\), as defined in \citet[Definition 2.5(i)]{Ansari-2021}, i.e.,
    \begin{align}\label{defgencopprod}
        C_{Y_n,Y_n'}(u,u') = C_{Y_n,X_n}\conv_{F_{X_n}} C_{Y_n,X_n}(u,u') = \int_0^1 \partial_2^{F_{X_n}} C_{Y_n,X_n}(t,u) \, \partial_2^{F_{X_n}} C_{Y_n,X_n}(t,u') \de t,
    \end{align}
    for \(u,u'\in [0,1],\)
    where \(\partial_2^{G}\) denotes a generalized partial derivative operator with respect to the second component and with respect to a distribution function \(G\); see \citet[Equation (2)]{Ansari-2021}.
    Then, due to \citet[Theorem 2.23]{Ansari-2021}, \(C_{Y_n,Y_n'}\) converges to the conditional independence product \(C_{Y,Y'} = C_{Y,X}\conv_{F_X} C_{Y,X}\), where we use assumptions \eqref{thmdel1} and \eqref{thmdel2}. This implies
    \begin{align}\label{eqthmdel3x}
        C_{Y_n,Y_n'} \xrightarrow{~\phantom{d}~} C_{Y,Y'} \quad \text{pointwise on } \overline{\Ran(F_Y)}\times \overline{\Ran(F_Y)}.
    \end{align}
    The convergence is also uniform by a similar reasoning to \eqref{equnifconvuv}.
    Now, consider \(z,z'\) such that \(F_Z\) is continuous at \(z\) and \(z'\). Then we have
    \begin{align}
        \label{eqthmdel5s} |F_{Z_n,Z_n'}(z,z') - F_{Z,Z'}(z,z')| &= | C_{Y_n,Y_n'}(F_{Z_n}(z),F_{Z_n}(z')) - C_{Y,Y'}(F_{Z}(z),F_{Z}(z')) | \\
        \label{eqthmdel5} &\leq | C_{Y_n,Y_n'}(F_{Z_n}(z),F_{Z_n}(z')) - C_{Y_n,Y_n'}(F_{Z}(z),F_{Z}(z')) | \\
        \label{eqthmdel5v} & ~~ + | C_{Y_n,Y_n'}(F_{Z}(z),F_{Z}(z')) - C_{Y,Y'}(F_{Z}(z),F_{Z}(z')) |, 
    \end{align}
    where the equality in \eqref{eqthmdel5s} holds true by Sklar's theorem and the invariance of copulas under increasing transformations; see \citet[Theorem 3.3]{Cai-2012}. 
    For the convergence of \eqref{eqthmdel5} to zero, we use Lipschitz continuity of copulas and \(Z_n\xrightarrow{d} Z\). For the convergence of \eqref{eqthmdel5v} to zero, we use the copula convergence in \eqref{eqthmdel3x} and \(\overline{\Ran(F_Y)} = \overline{\Ran(F_Z)}.\)
\end{proof}

The following result is a consequence of Theorem \ref{thmdel} for stochastically increasing/decreasing bivariate copulas. 
In this case, \(\partial_2\)-convergence and pointwise convergence are equivalent; see e.g. \citet[Proposition 3.6]{Siburg-2021}.
Recall that a bivariate copula \(C\) is said to be \emph{stochastically increasing/decreasing (SI/SD)} if, for all \(v\in [0,1],\) the function \(t\mapsto \partial_2 C(v,t)\) is non-increasing/non-decreasing outside a \(\lambda\)-null set. 
%for Lebesgue-almost all \(t\in (0,1).\) 
We write \((V,U)\sim C\) for a random vector having distribution function \(C.\)

\begin{corollary}[Convergence for SI random vectors]\label{corSI}~\\
    Let \((C_n)_{n\in \N}\) be a sequence of bivariate SI/SD copulas that converges pointwise to a bivariate copula \(C.\) Then, for random vectors \((V_n,U_n)\sim C_n\) and \((V,U)\sim C,\) we have \(\xi(V_n,U_n) \to \xi(V,U).\)  
\end{corollary}

\begin{proof}
Since, for a sequence of SI/SD copulas, pointwise convergence and \(\partial_2\)-convergence are equivalent, condition \eqref{thmdel1} of Theorem \ref{thmdel} is fulfilled. Since \(V_n,U_n,V,U\) are all uniformly distributed on \((0,1),\) conditions \eqref{thmdel2} and \eqref{thmdel3} are trivially satisfied.
\end{proof}

\begin{remark}\label{remfre}
\begin{enumerate}[(a)]
\item Condition \eqref{thmdel1} of Theorem \ref{thmdel} is used for the convergence of the conditional distribution functions occurring in the integrand of \eqref{defgencopprod}, noting that, for fixed \(y\in \R,\) we have \(\partial_2^{F_X} C_{Y,X}(F_Y(y),F_X(x)) = F_{Y|X=x}(y)\) for \(P^X\)-almost all \(x\); see \citet[Theorem 2.2]{Ansari-2021}.
%ensures that, for all \(u\in [0,1],\) the conditional probabilities \(P(Y_n \leq F_{Y_n}^{-1}(u) \mid X_n = F_{X_n}^{-1}(t))\) converge to \(P(Y \leq F_Y^{-1}(u) \mid X = F_X^{-1}(t))\) for Lebesgue-almost all \(t\in [0,1].\)
Some well-known approximations of copulas satisfying the concept of \(\partial_2\)-convergence are the checkerboard, check-min or Bernstein approximation of copulas; see \citet[Example 4]{Mikusinski-2010}. In contrast, shuffles of min discussed in Example \ref{Cont.ExWC} do not converge with respect to the \(\partial_2\)-convergence.
Related concepts for convergence of conditional distributions are studied by \cite{sfx2021weak}. 
In contrast to Theorem \ref{Cont.MainThm}, the assumptions in Theorem \ref{thmdel} do not imply weak continuity of \(X_n \to X\) or \(Y_n \to Y\); see Remark \ref{RemCont} \eqref{RemCont:2}.
%condition in the conditioning variable nor weak convergence of the marginal distributions. 
    \item Condition \eqref{thmdel2} in Theorem \ref{thmdel} on range convergence of \(F_{X_n}\) is used for the pointwise convergence of the copulas \(C_{Y_n,Y_n'}\). 
    Note that the copula \(C_{Y_n,Y_n'}\) in \eqref{defgencopprod} is given through a generalization of the \(\conv\)-product in \eqref{defcopprod} allowing for mixing random variables that may also have a discontinuous distribution function \(F_{X_n}.\) 
    %For the pointwise convergence of the generalized \(\conv\)-product, we use condition \eqref{thmdel2} in Theorem \ref{thmdel} on range convergence of \(F_{X_n}.\)
    The standard Markov product of copulas in \eqref{defcopprod} is only applicable in the case where \(F_{X_n}\) and \(F_X\) are continuous; see also Example \ref{exrangconvX}.
    %\item \label{remfre2} Measures of sensitivity ... 
    \item Conditions \eqref{thmdel2} and \eqref{thmdel3} of Theorem \ref{thmdel} are satisfied, in particular, if the random vectors \((Y_n,X_n)\) and \((Y,X)\) are from the same Fr{\'e}chet class, i.e., \(F_{X_n} = F_X\) and \(F_{Y_n} = F_Y\) for all \(n\in \N.\) This is the case in the setting of Corollary \ref{corSI}. 
    In the proof of Theorem \ref{thmdel}, we use that the Markov products of the transformed vectors in \eqref{eqthmdel1} converge weakly, whereas, in general, \((Y_n,Y_n') \not\xrightarrow[]{d}(Y,Y')\).
    %\item The copula product in \eqref{defgencopprod} is a generalization of the \(\conv\)-product in \eqref{defcopprod} to mixing random variables that may also have a discontinuous distribution function. In the proof of Theorem \ref{thmdel}, we use condition \eqref{thmdel2} on range convergence of \(F_{X_n}\) to \(F_X\) to obtain convergence of the copulas in \eqref{eqthmdel3}. 
    %\item For a bivariate random vector \((U,V)\) with distribution function \(C,\) the SI/SD property of \(C\) is equivalent to the conditional distribution \(V|U=u\) being increasing in \(u\) with respect to the stochastic order, see \citet[Theorem 5.2.10]{Nelsen-2006}.
    \item The SI/SD assumptions in Corollary \ref{corSI} are positive/negative dependence concepts satisfied by various copulas families. For an overview of many well-known bivariate copula families, that are SI/SD and pointwise continuous in the underlying copula parameter; see \cite{Ansari-Rockel-2023}.
\end{enumerate}
    
\end{remark}

In the following example, we discuss the relevance of condition \eqref{thmdel2} in Theorem \ref{thmdel}.

\begin{example}[Range convergence of \(F_{X_n}\)]\label{exrangconvX}~\\
    Let \((Y_n,X_n)_{n\in \N}\) be a sequence of random variables defined by \((Y_n,X_n) = (U,q_n(U)),\) where \(U\) is uniform on \((0,1)\) and where \(q_n\) denotes the quantile function of the zero mean normal distribution with variance \(1/n.\) 
    Then, we have that \((Y_n,X_n)\xrightarrow{~d~} (U,0) =: (Y,X).\)
    The uniquely determined copula of \((Y_n,X_n)\) is given by the comonotonicity copula, i.e., \(C_{Y_n,X_n}(u,v) = M(u,v) := \min\{u,v\}\) for all \(u,v\in [0,1].\) Since \(X\) follows a Dirac-distribution, \(M\) is also a copula of \((Y,X),\) so that condition \eqref{thmdel1} in Theorem \ref{thmdel} is fulfilled. Further, condition \eqref{thmdel3} is also trivially satisfied.
    However, condition \eqref{thmdel2} is not satisfied since \((F_{X_n}\circ F_{X_n}^{-1})(t) = t \ne 1 = (F_X \circ F_X^{-1})(t)\) for all \(t\in (0,1).\) 
    The following facts show that condition \eqref{thmdel2} in Theorem \ref{thmdel} on range convergence of \(F_{X_n}\) cannot be omitted:
    \begin{enumerate}[(a)]
        \item Since \(Y_n = F_{X_n}(X_n)\) for all \(n,\) we have that \(\xi(Y_n,X_n) = 1\ne 0 =\xi(Y,X)\) for all \(n\) using for the last equality that \(X\) follows the Dirac distribution in \(0.\) Hence, \(\xi(Y_n,X_n)\) does not converge to \(\xi(Y,X).\)
        \item Since \(Y_n\) perfectly depends on \(X_n\) for all \(n,\) we have that \(Y_n = Y_n' = U\) almost surely for all \(n.\)  It follows that \(C_{Y_n,Y_n'} = M\) for all \(n\) noting that the copula of \((Y_n,Y_n')\) is uniquely determined.
        However, for the Markov product of \((Y,X),\) we obtain that \(Y\) and \(Y'\) are independent because \(X\) follows a Dirac distribution. Note that, in this case, conditional independence of \(Y\) and \(Y'\) given \(\XX\) is equivalent to independence of \(Y\) and \(Y'.\)
        The uniquely determined copula of \((Y,Y')\) is the independence copula \(\Pi(u,v) := uv.\) Consequently, \(C_{Y_n,Y_n'} = M\) does not converge to \(C_{Y,Y'} = \Pi.\) 
    \end{enumerate}
\end{example}

\section{Continuity of dependence measures related to \(\xi\)}\label{secreldepm}

In this section, we show that our continuity results apply to various dependence measures recently studied in the literature.
In Section \ref{secmultT}, we consider an extension of Chatterjee's rank correlation to a vector of multi-output variables which we focus on in Section \ref{Sec.Cont.PM} in more detail. In Section \ref{secmeaexpl}, we establish continuity of a measure of explainablity.
Section \ref{secreldepme} then covers an overview of several dependence measures which apply to our continuity results in Section \ref{sec2}.

\subsection{An extension of \(\xi\) to multivariate output variables}\label{secmultT}

An extension of Chatterjee's rank correlation $\xi$ to a multi-response vector $\YY = (Y_1, \dots, Y_q)$, $q \in \N$, has been recently proposed by \cite{ansari2023MFOCI}. The extension is defined by 
\begin{align} \label{defmdm}
  T (\YY,\XX)
  & := 1 - \frac{q - \sum_{i=1}^{q} \xi(Y_i, (\XX,Y_{i-1},\dots,Y_{1}))}{q - \sum_{i=1}^{q} \xi(Y_i, (Y_{i-1},\dots,Y_{1}))}, 
   \qquad \text{with}~~~ \xi(Y_1,\emptyset):=0,
\end{align}
noting that, for \(q=1,\) \(T\) reduces to \(\xi.\)
%\(T\) is a true generalization of $\xi$ in that $T(Y,\XX) = \xi(Y,\XX)$ for $Y$ being univariate.
The measure \(T\) inherits the basic properties of \(\xi,\) i.e., $0 \leq T(\YY,\XX) \leq 1$, where $T(\YY,\XX) = 0$ if and only if $\YY$ and $\XX$ are independent, and $T(\YY,\XX) = 1$ if and only if $\YY$ perfectly depends on $\XX$.
The following continuity result for the measure $T$ defined by \eqref{defmdm} is an immediate consequence of Theorem \ref{ThmCont} and Theorem \ref{Cont.MainThm}.

\begin{corollary}[Continuity of \(T\)]~~\label{Cont.MainThm.Multi}
Let \((\YY,\XX)\) be a $(q+p)$-dimensional random vector and let 
$(\YY_n,\XX_n)_{n\in\N}$ be a sequence 
%\linebreak $=(X_{1,n},\ldots,X_{p,n},Y_{1,n},\ldots,Y_{q,n})_{n\in\N}$
of \((q+p)\)-dimensional random vectors. 
Let \(V_1 \subset \R^{p}\) be open such that \(P(\XX\in V_1)=1\) and,
for all \(i\in \{2,\ldots,q\}\), 
let \(U_i\subset \R^{i-1}\) and \(V_i\subset \R^{p+i-1}\) be open such that \(P((Y_{1},\ldots,Y_{i-1})\in U_i)=1\) and \(P((Y_{1},\ldots,Y_{i-1},\XX)\in V_i)=1\).
Consider the following conditions:
\begin{enumerate}[(i)]
\item\label{Cont.MainThm.Multi.1} 
\((\YY_n,\XX_n)\xrightarrow{ d } (\YY,\XX),\)
\item\label{Cont.MainThm.Multi.2} 
\((\E[u(Y_{1,n})|\XX_n=\,\cdot\,])_{n\in \N}\) is asymptotically equicontinuous on \(V_1,\) \\
\(\left(\E[u(Y_{i,n})|(Y_{1,n},\ldots,Y_{i-1,n},\XX_n)=\,\cdot\,]\right)_{n\in \N}\) is asymptotically equicontinuous on \(V_i,\) and \\ 
\(\left(\E[u(Y_{i,n})|(Y_{1,n},\ldots,Y_{i-1,n})=\,\cdot\,]\right)_{n\in \N}\) is asymptotically equicontinuous on \(U_i\) f.a. \(i\in \{2,\ldots,q\}\) and \(u\in \cU(\R),\)
\item\label{Cont.MainThm.Multi.12}
The Markov product of \((Y_{1,n},\XX_{n})\) converges weakly to the Markov product of \((Y_{1},\XX)\), \\
the Markov product of \((Y_{i,n},(Y_{1,n},\dots,Y_{i-1,n},\XX_n))\) converges weakly to the Markov product of \((Y_{i},(Y_{1},\dots,Y_{i-1},\XX))\), and 
the Markov product of \((Y_{i,n},(Y_{1,n},\dots,Y_{i-1,n}))\) converges weakly to the Markov product of \((Y_{i},(Y_{1},\dots,Y_{i-1}))\) for all \(i \in \{2,\dots,q\}\),
\item\label{Cont.MainThm.Multi.3} \((F_{Y_{i,n}}\circ F_{Y_{i,n}}^{-1})(t) \xrightarrow[n\to \infty]{} (F_{Y_i}\circ F_{Y_i}^{-1})(t)\) for \(\lambda\)-almost all \(t\in (0,1)\) and for all \(i\in \{1,\ldots,q\}\).
\end{enumerate}
Then the following assertions hold true.
\begin{enumerate}[(a)]
    \item \eqref{Cont.MainThm.Multi.1} + \eqref{Cont.MainThm.Multi.2} \(\Longrightarrow\) \eqref{Cont.MainThm.Multi.12}
    \item \eqref{Cont.MainThm.Multi.12} + \eqref{Cont.MainThm.Multi.3} \(\Longrightarrow\) \(\lim_{n \to \infty} T(\YY_n,\XX_n) = T(\YY,\XX).\)
\end{enumerate}
\end{corollary}

In Section \ref{Sec.Cont.PM}, we discuss the measure \(T\) in more detail and verify continuity of \(T\) for the class of elliptical and \(\ell_1\)-norm symmetric distributions.

\subsection{A measure of explainability}\label{secmeaexpl}

For a square integrable random variable \(Y\), consider the functional \(\Lambda(Y,\XX) := \mathbb{V}(\mathbb{E}(Y \vert \XX))/\mathbb{V}(Y).\)
Then \(\Lambda(Y,\XX)\) coincides with the fraction of explained variance of $Y$ given $\XX$, also known as the Sobol' index 
%or Cram{\'e}r-von-Mises index 
(see \cite{Sobol-1993}). It can be represented as
\begin{eqnarray}\label{Explain.Rep.}
  \Lambda(Y,\XX) 
  &  = & \varrho(Y,Y')
     = \frac{\mathbb{E} (Y\,Y') - \mathbb{E}(Y)\,\mathbb{E}(Y')}{\sqrt{\mathbb{V}(Y)}\sqrt{\mathbb{V}(Y')}},
\end{eqnarray}
where $\varrho$ denotes Pearson's correlation coefficient.
The representation of \(\Lambda(Y,\XX)\) as the Pearson's correlation of the Markov product \((Y,Y')\) can be found in \cite{janon2014,sfx2022phi}. According to \cite{Ansari-LFT-2023}, \(\Lambda(Y,\XX)\) measures the sensitivity of the conditional expectations \(\mathbb{E}(Y \vert \XX)\).
Recall that \(\Lambda(Y,\XX) \in [0,1]\), and we have  
\(\Lambda(Y,\XX) = 0\) if and only if \(\mathbb{V}(\mathbb{E}(Y \vert \XX)) = 0\) (which is the case if \(Y\) and \(\XX\) are independent, but not vice versa), and 
\(\mathbb{E}(Y \vert \XX) = 1\) if and only if \(Y\) perfectly depends on \(\XX\).

The following result shows that \(\Lambda\) is continuous with respect to weak convergence of the Markov product. Hence, our continuity results for Markov products from Section \ref{sec2} yield sufficient conditions for continuity of \(\Lambda.\)

\begin{corollary}[Continuity of $\Lambda$]~~\label{CorContLambda}\\
For random vectors \((Y_n,\XX_n)\) and \((Y,\XX)\), let \(Y'\) and, similarly, \(Y_n'\) be defined as in \eqref{Assumption.DimR}. 
If 
\begin{enumerate}[(i)]
\item \label{CorContLambda1} \((Y_n,Y_n') \xrightarrow{~d~} (Y,Y')\), and 
\item \label{CorContLambda2} \(\sup_{n \in \mathbb{N}} \mathbb{E}(|Y_n|^{2+\varepsilon}) < \infty\) for some $\varepsilon > 0$,
\end{enumerate}
then \(\lim_{n \to \infty} \Lambda(Y_n,\XX_n) = \Lambda(Y,\XX).\)    
\end{corollary}
\begin{proof}
Convergence of \(\mathbb{E}(Y_n)\), \(\mathbb{E}(Y_n')\),\(\mathbb{V}(Y_n)\), and \(\mathbb{V}(Y'_n)\) in \eqref{Explain.Rep.} follows by uniform integrability in \eqref{CorContLambda2} and \citet[Theorem 3.5]{Billingsley-1999}, and weak convergence of the product \(Y_n\,Y_n'\) to \(Y\,Y'\) is due to the continuous mapping theorem. It remains to prove uniform integrability of \(Y_n\,Y_n'\) which then implies the convergence of \(\mathbb{E} (Y_n\,Y'_n)\) to \(\mathbb{E} (Y\,Y')\). But this follows from \(\mathbb{E}(|Y_n\,Y_n'|^{1+\varepsilon/2}) \leq  \mathbb{E}(|Y_n|^{2+\varepsilon})\) using Cauchy-Schwarz inequality and \(Y_n \eqd Y_n'.\) 
Thus, \(\lim_{n \to \infty} \Lambda(Y_n,\XX_n) = \Lambda(Y,\XX).\)
\end{proof}

\subsection{Related dependence measures}\label{secreldepme}

In the following, we briefly discuss continuity of related dependence measures studied in the literature.

As already mentioned, the kernel partial correlation coefficient in \citet[Proposition 8]{deb2020b} is based on similar continuity conditions as the measure of explainability in Corollary \ref{CorContLambda}.
Hence, our results on weak continuity of Markov products in Section \ref{sec2}, particularly Theorem \ref{Cont.MainThm}, also apply to the kernel partial correlation.

The optimal transport-based Wasserstein correlation coefficient studied by \cite{wiesel2022} underlies similar modes of convergence. 
Due to \citet[Theorem 4.1]{wiesel2022}, it is continuous with respect to the adapted Wasserstein distance, which also accounts for the distance between conditional distributions. 
Since optimal transport plans are stable under weak convergence \citet[Theorem 5.20]{Villani-2009}, a version of Theorem \ref{Cont.MainThm} on conditional weak convergence also applies to the Wasserstein correlation coefficient.

The measures of sensitivity in \cite{Ansari-LFT-2023} are defined by comparing conditional distribution functions. The continuity result in \citet[Proposition 8]{Ansari-LFT-2023} makes use of the copula-based concept of weak conditional convergence due to \cite{sfx2021weak}. Using the continuity results for generalized copula products in \citet[Theorem 2.23]{Ansari-2021}, it can be shown that these measures are also continuous under the assumptions of Theorem \ref{thmdel}. This also holds true for the copula-based dependence measures in \cite{fgwt2020} which compare the distance of conditional distribution functions with respect to independence.

%\clearpage
%%%%%%%%%%%%%%%%%%%%%%%%%%%%%%%%%%%%%%%%%%%%%%%%%%%%%%%%%%%%%%%%%%%%%%%%%%%%%%%%%%%%%%%%
\section{Continuity results for families of distributions}
\label{Sec.Cont.PM}

In this section, we show for well-known families of multivariate distributions weak continuity of their Markov products. 
For \(\ell_1\)-norm symmetric distributions, we verify the conditions in Theorem \ref{ThmCont} directly. For the class of elliptical distributions, we establish sufficient conditions for conditional weak convergence due to Theorem \ref{Cont.MainThm}.
The distributions we focus on all exhibit continuous marginal distribution functions so that range continuity is trivially satisfied.
As a consequence, we obtain continuity of \(\xi\) and of related dependence measures as discussed in the previous section. 
Our focus in this section is on the multi-output extension \(T,\) for which we discuss its continuity properties in more detail.
%\JA{For better readability, the proofs of all the results in this section are deferred to the appendix.}

%%%%%%%%%%%%%%%%%%%%%%%%%%%%%%%%%%%%%%%%%%%%%%%%%%%%%%%%%%%%%%%%%%%%%%%%%%%%%%%%%%%%%%%%
\subsection{Continuity within the class of elliptical distributions} \label{Sec.Cont.PM.Ellipt.}

A \((q+p)\)-dimensional random vector \((\YY,\XX)\) is said to be \emph{elliptically distributed}, written \((\YY,\XX)\sim \cE({\boldsymbol \mu},\Sigma,\phi),\) for some vector \({\boldsymbol \mu}\in \R^{q+p},\) 
for some positive semi-definite matrix \(\Sigma=(\sigma_{ij})_{1\leq i,j\leq q+p},\) and 
some generator \(\phi\colon \R_+ \to \R,\) if the characteristic function of \((\YY,\XX)-{\boldsymbol \mu}\) is the function \(\phi\) applied to the quadratic form \({\bf t}^T\Sigma {\bf t},\) i.e., \(\varphi_{(\YY,\XX)-{\boldsymbol \mu}}({\bf t})=\phi({\bf t}^T\Sigma {\bf t})\) for all \({\bf t}\in \R^{q+p}.\) 
%W.l.o.g., we assume that \(\XX\) and \(\YY\) have nondegenerate components, i.e., \(\sigma_{ii}>0\) for all \(i\in \{1,\ldots,p+q\}.\)
For example, if \(\phi(u)=\exp(-u/2),\) then \((\YY,\XX)\) is multivariate normal with mean vector \({\boldsymbol \mu}\) and covariance matrix \(\Sigma.\) 
Elliptical distributions have a stochastic representation 
\begin{align}\label{radialpart}
(\YY,\XX)\eqd \boldsymbol{\mu} + R A \UU^{(k)},
\end{align}
where \(R\) is a non-negative random variable, \(A^TA=\Sigma\) is a full rank factorization of \(\Sigma,\) and where \(\UU^{(k)}\) is uniform on the unit sphere in \(\R^k\) with \(k=\rank(\Sigma);\) 
see \cite{Cambanis-1981} and \cite{FangKotz-1990} for several properties of elliptical distributions. 
%For \(i\ne j,\) denote by \(\varrho_{ij}\) the correlation between the \(i\)-th and \(j\)-the component of the vector \((\XX,\YY),\) i.e., \(\varrho_{ij}=\frac{\sigma_{ij}}{\sqrt{\sigma_{ii}}\sqrt{\sigma_{jj}}}\) whenever 
We will need the decomposition 
\begin{align}\label{eqdecsig}
\Sigma = \begin{pmatrix}
\Sigma_{11} & \Sigma_{12} \\
\Sigma_{21} & \Sigma_{22}
\end{pmatrix}
\end{align}
and use that elliptical distributions are closed under marginalization (i.e., also their marginal distributions are elliptical). More precisely, \(\Sigma_{11}\) is of dimension \(q\times q\), and we have that \(\YY\sim \cE({\boldsymbol \mu}_1,\Sigma_{11},\phi)\) and \(\XX\sim \cE({\boldsymbol \mu}_2,\Sigma_{22},\phi)\), where \({\boldsymbol \mu}=({\boldsymbol \mu}_1,{\boldsymbol \mu}_2)\) with \({\boldsymbol \mu}_1\in \R^q\) and \({\boldsymbol \mu}_2\in \R^p\); see e.g. \citet[Corollary 1 of Theorem 2.6.3]{Fang-1990}.

The following result verifies the continuity conditions in Corollary \ref{Cont.MainThm.Multi} for elliptical distributions. 
In particular, it provides sufficient conditions for continuity of $\xi$ and \(T\) in the scale matrix \(\Sigma\) and the radial part \(R.\) Note that $\xi$ and \(T\) are location-scale invariant, and thus, they neither depend on the centrality parameter \(\boldsymbol{\mu}\) nor on componentwise scaling factors \citep[Theorem A.2]{ansari2023MFOCI}.

\begin{theorem}[Continuity for elliptical distributions]\label{theconell}~\\
Let \((\YY_n,\XX_n)\sim \cE({\boldsymbol{\mu}_n},\Sigma_n,\phi_n),\) \(n\in \N,\) and \((\YY,\XX)\sim \cE(\boldsymbol{\mu},\Sigma,\phi)\) be \((q+p)\)-dimensional elliptically distributed random vectors.
Assume that \(\Sigma_{n}\), \(n\in \N\), and \(\Sigma\) are positive definite. If \(\Sigma_n\to \Sigma\) componentwise and if either
\begin{enumerate}[(i)]
    \item \label{theconell1} \(\phi_n =\phi\)  for all \(n,\) and the radial part \(R\) in \eqref{radialpart} associated with \(\phi\) has a continuous distribution function, or
    \item \label{theconell2} \(\phi_n(u) \to \phi(u)\) for all \(u\geq 0,\) and the radial variable \(R_n\) associated with \(\phi_n\) has a density \(f_n\) such that
    \begin{enumerate}
        \item \((f_n)_{n\in \N}\) is asymptotically uniformly equicontinuous on \((0,\infty),\) and
        \item \((f_n)_{n\in \N}\) is pointwise bounded, i.e., \(M(x):=\sup_{n\in \N} f_n(x)<\infty\) for all \(x\in (0,\infty),\)
        %\item \(M(x)\to 0\) as \(x\to \infty,\)
    \end{enumerate}
\end{enumerate}
then conditions \eqref{Cont.MainThm.Multi.1} -- \eqref{Cont.MainThm.Multi.3} in Corollary \ref{Cont.MainThm.Multi} are satisfied and thus \(T(\YY_n,\XX_n) \to T(\YY,\XX).\)
\end{theorem}

The proof of Theorem \ref{theconell} is deferred to the appendix. 
As an application of Theorem \ref{theconell}\eqref{theconell1}, we immediately obtain the following continuity result for the multivariate normal distribution.

\begin{corollary}[Continuity for normal distributions]~\\
Let \((\YY_n,\XX_n)\sim N(\boldsymbol{\mu}_n,\Sigma_n),\) \(n\in \N,\) and \((\YY,\XX)\sim N(\boldsymbol{\mu},\Sigma)\) be \((q+p)\)-dimensional normally distributed random vectors with \(\Sigma_n,\) \(n\in \N,\) and \(\Sigma\) positive definite. 
Assume that \(\Sigma_n\to \Sigma\) componentwise. 
Then, conditions \eqref{Cont.MainThm.Multi.1} -- \eqref{Cont.MainThm.Multi.3} in Corollary \ref{Cont.MainThm.Multi} are satisfied and thus \(T(\YY_n,\XX_n) \to T(\YY,\XX).\)
\end{corollary}

Denote by \(t_\nu(\boldsymbol{\mu},\Sigma)\) the \(d\)-variate Student-t distribution with \(\nu>0\) degrees of freedom, symmetry vector \(\boldsymbol{\mu}\in \R^d\) and symmetric, positive semi-definite \((d\times d)\)-matrix \(\Sigma.\) Then \(t_\nu(\boldsymbol{\mu},\Sigma)\) belongs to the elliptical class, where the radial variable \(R\) has a density of the form \(g(t)= c[1+t/\nu]^{-(\nu+d)/2},\) which is Lipschitz-continuous with Lipschitz constant \((\nu+d)/(2\nu).\) Hence, the following result is an application of Theorem \ref{theconell}\eqref{theconell2}.

\begin{corollary}[Continuity for Student-t distributions]~\\
Let \((\YY_n,\XX_n)\sim t_{\nu_n}(\boldsymbol{\mu}_n,\Sigma_n),\) \(n\in \N,\) and \((\YY,\XX)\sim t_\nu(\boldsymbol{\mu},\Sigma)\) be \((q+p)\)-dimensional Student-t distributed random vectors with \(\Sigma_n,\) \(n\in \N,\) and \(\Sigma\) positive definite. 
Assume that \(\Sigma_n\to \Sigma\) componentwise and \(\nu_n\to \nu.\) 
Then, conditions \eqref{Cont.MainThm.Multi.1} -- \eqref{Cont.MainThm.Multi.3} in Corollary \ref{Cont.MainThm.Multi} are satisfied and thus \(T(\YY_n,\XX_n) \to T(\YY,\XX).\)
\end{corollary}

For studying the behavior of \(\xi\) and \(T\) in elliptical models, the extreme elements are of particular importance. The following theorem characterizes for the class of elliptical distributions the cases where the measure \(T\) attains the values \(0\) and \(1,\) respectively. 
%For the proof, we make use of stochastic representations of conditional elliptical distributions.

\begin{proposition}[Characterization of extremal cases in elliptical models]\label{thenormodextcas}~\\
Let \((\YY,\XX)\sim \cE({\boldsymbol \mu},\Sigma,\phi)\) with non-degenerate components.
Then for \(\Sigma\) decomposed by \eqref{eqdecsig}, it holds that
\begin{enumerate}[(i)]
\item \label{thenormodextcas1}
\(T(\YY,\XX)=0\) if and only if \(\Sigma_{21}\) is the null matrix (i.e., \(\sigma_{ij}=0\) for all \((i,j)\in \{q+1,\ldots,q+p\} \times \{1,\ldots,q\}\)) and \((\YY,\XX)\) is multivariate normally distributed.
\item \label{thenormodextcas2}
\(T(\YY,\XX)=1\) if and only if \(\rank(\Sigma)=\rank(\Sigma_{22}).\)
\end{enumerate}
\end{proposition}

The proof of Proposition \ref{thenormodextcas} is given in the appendix.
For \(q=1,\) the following example illustrates the behavior of $\xi$ in dependence of the parameter of the equicorrelated multivariate normal distribution.

\begin{figure}[t!]
		\centering  
		\includegraphics[width=0.8\textwidth,trim = 10 46 10 57, clip]{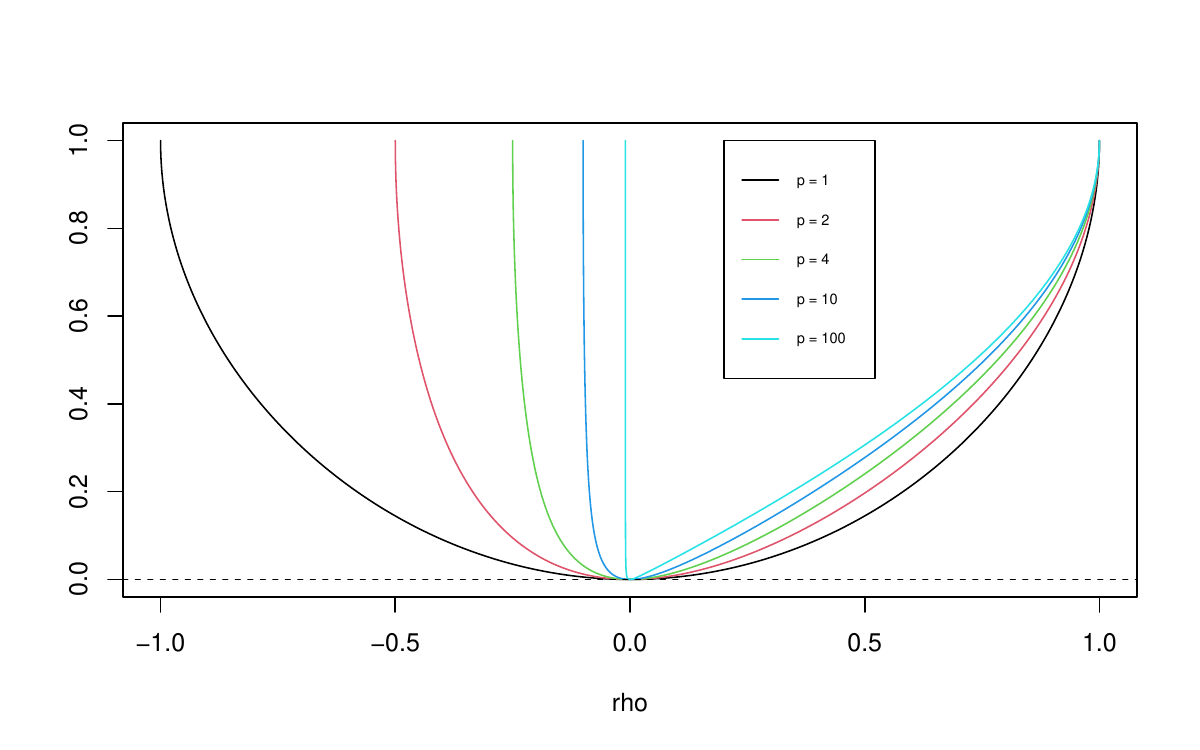}
        \caption{Plots of \(\xi(Y,\XX)\) in dependence on the correlation \(\varrho\) in the equicorrelated normal setting of Example \ref{exeqnor} for various dimensions \(p\in \{1,2,4,10,100\}\) of \(\XX.\)}
		\label{Fig.eqcor_normal}
\end{figure}	

\begin{example}[Equicorrelated normal distribution]\label{exeqnor}~\\
Chatterjee's rank correlation exhibits a closed-form expression for the multivariate normal distribution. For a \((1+p)\)-dimensional random vector \((Y,\XX)\sim N({\bf 0},\Sigma)\) with covariance matrix \(\Sigma = \left(\begin{smallmatrix} 
  \sigma_Y^2 & \Sigma_{12} \\
  \Sigma_{21} &  \Sigma_{22} \end{smallmatrix}\right)\) for \(\sigma_Y>0,\) it holds that
\begin{align}\label{eqpropChatformmGau}
  \xi(Y,\XX) = \frac 3 \pi \arcsin\left(\frac{1+r^2}{2}\right)-\frac 1 2 , 
  \qquad \text{with } r = \sqrt{\Sigma_{12} \Sigma_{22}^{-} \Sigma_{21} / \sigma_Y^2},
\end{align}
where \(\Sigma_{22}^-\) denotes a generalized inverse of \(\Sigma_{22}\) such as the Moore–Penrose inverse; see \citet[Proposition 2.7]{ansari2023MFOCI}.
Assume that \(\Sigma=(\sigma_{ij})\) is equicorrelated with correlation \(\varrho,\) i.e., \(\sigma_{ij}=1\) for \(i = j\) and \(\sigma_{ij} = \varrho\) for \(i\ne j.\) Then \(\Sigma\) is positive semidefinite if and only if \(\varrho\in [-1/p,1].\) The parameter \(r\) in \eqref{eqpropChatformmGau} is given by \(r = \varrho \sqrt{\frac{p}{1 + (p-1) \varrho}}\).
Figure \ref{Fig.eqcor_normal} illustrates the value of \(\xi(Y,\XX)\) in dependence on \(\varrho\) for various dimensions \(p.\) In accordance with Proposition \ref{thenormodextcas}, we observe for all \(p\) that \(\xi(Y,\XX) = 0\) if and only if \(\varrho = 0.\) Further, \(\xi(Y,\XX) = 1\) if and only if \(\varrho\in \{-1/p,1\}.\) 
In the latter case, we have the perfect (linear) dependence \(Y = -\sum_{i=1}^p X_i\) for \(\varrho = -1/p,\) and \(Y = \tfrac 1 p \sum_{i=1}^p X_i\) for \(\varrho = 1.\)
\end{example}

The following example studies continuity of \(T\) in the case of a \(4\)-dimensional normal / Student-t distribution for a bivariate response vector and a bivariate predictor vector (i.e., \(q=2\) and \(p=2\)).

\begin{example}[\(4\)-dimensional normal / Student-t distribution]\label{4dnormal}~\\
Assume that \((\YY,\XX)=(Y_1,Y_2,X_1,X_2)\)
follows a \(4\)-dimensional normal distribution with the three-parametric covariance matrix \(\Sigma\) given by
\begin{align}\label{eq4dnormal}
    \Sigma = \begin{pmatrix}
        1 & \varrho_Y & \varrho_{YX} & \varrho_{YX} \\
        \varrho_Y & 1 & \varrho_{YX} & \varrho_{YX} \\
        \varrho_{YX}& \varrho_{YX} & 1 & \varrho_X  \\
        \varrho_{YX}& \varrho_{YX} & \varrho_X & 1 
    \end{pmatrix}
\end{align} 
for some correlation parameters \(\varrho_Y,\varrho_{YX},\varrho_X\in [-1,1].\) Elementary calculations show that \(\Sigma\) is positive semi-definite if and only if 
\(\varrho_Y \in [-1,1],\) \(\varrho_X \in [-1,1]\) and \(\varrho_{YX}^2 
   \leq \tfrac{1+\varrho_Y}{2}\tfrac{1+\varrho_X}{2}\).
For the multivariate normal distribution, we have the closed-form expression
\begin{align*}
    T(\YY,\XX) 
    %&= 1- \frac{2- T(Y_1|\XX)-T(Y_2|(Y,\XX_1))}{2-T(Y_2|Y_1)}\\
    & = 1- \frac{3- \frac{3}{\pi} \left[\arcsin\left(\frac 1 2 + \frac{\varrho_{YX}^2}{1+\varrho_X}\right)+\arcsin\left(\frac 1 2 + \frac{(1+\varrho_X)\varrho_Y^2-2 (2\varrho_Y-1) \varrho_{YX}^2}{2(1+\varrho_X)-4\varrho_{YX}^2}\right)\right]}{\frac{5}{2}-\frac{3}{\pi}\arcsin\left(\frac{1+\varrho_Y^2}{2}\right)};
\end{align*}
see \citet[Example A.7]{ansari2023MFOCI}.
Figure \ref{Fig.Tq_Gauss-t} illustrates continuity of \(T(\YY,\XX)\) in the correlation parameter \(\varrho_{YX}\) for fixed correlation \(\varrho_X=0.5\) and for several choices of \(\varrho_Y\) in the \(4\)-dimensional normal model (solid lines) and for \((\YY,\XX)\) following a Student-t distribution with scale matrix \(\Sigma\) in \eqref{eq4dnormal} and \(3\) degrees of freedom (dotted lines). We observe that, for fixed \(\varrho_X\) and \(\varrho_Y,\) the parameter constraints due to positive semi-definiteness of \(\Sigma\) generally restrict the range of \(T.\)  Further, the values of \(T\) almost coincide for the normal and the Student-t distribution with \(3\) degrees of freedom. It can be seen from the plots that \(T(\YY,\XX)\) converges to \(0\) for \(\varrho_{YX}\to 0\) only for the normal distribution, which confirms Proposition \ref{thenormodextcas} \eqref{thenormodextcas1}. By the specific choice of the covariance matrix \(\Sigma,\) \(T(\YY,\XX)\) converges to \(1\) if and only if \(\varrho_Y \nearrow 1\) and \(\varrho_{YX}^2\nearrow \frac 3 4,\) using that \(\varrho_X=0.5\). In this limit case, we have \(\rank(\Sigma) = \rank(\Sigma_{22}),\) which confirms Proposition \ref{thenormodextcas} \eqref{thenormodextcas2}.
\end{example}

\begin{figure}[t!]
		\centering
        \includegraphics[width=0.8\textwidth,trim = 25 46 10 57, clip]{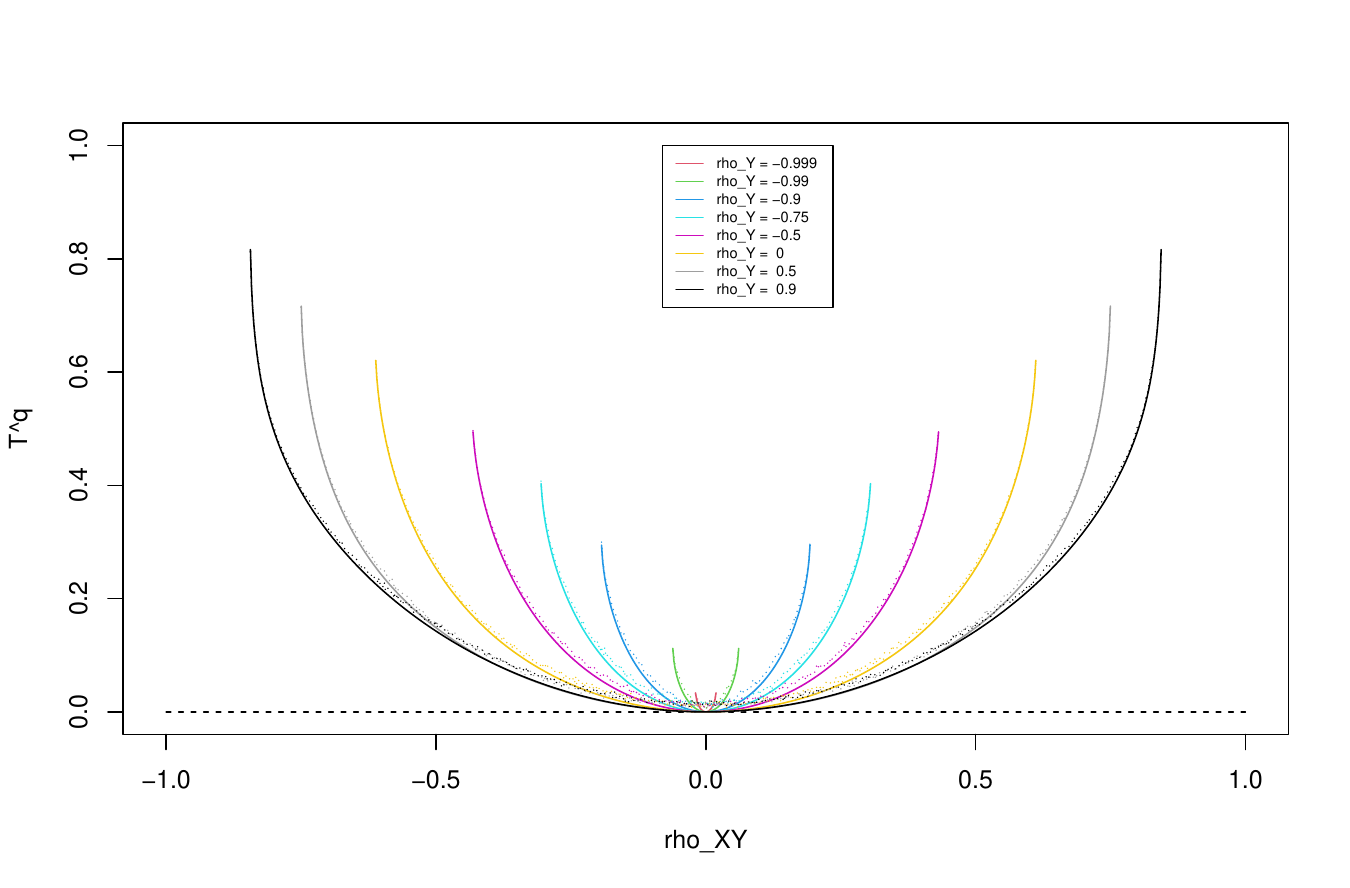}
  \caption{Plots of \(T(\YY,\XX)\) for \((\YY,\XX)=(Y_1,Y_2,X_1,X_2)\) being normally distributed (solid lines) and \(t\)-distributed with \(3\) degrees of freedom (dotted) with covariance matrix \(\Sigma\) given by \eqref{eq4dnormal} in dependence on \(\varrho_{YX}\) for fixed \(\varrho_X=0.5\) and for the choices \(\varrho_{Y}\in \{-0.999,-0.99,-0.9,-0.75,-0.5,0,0.5,0.9\}.\)}
		\label{Fig.Tq_Gauss-t}
\end{figure}

%%%%%%%%%%%%%%%%%%%%%%%%%%%%%%%%%%%%%%%%%%%%%%%%%%%%%%%%%%%%%%%%%%%%%%%%%%%%%%%%%%%%%%%%
\subsection{Continuity within the class of \(\ell_1\)-norm symmetric distributions} \label{Sec.Cont.PM.L1Norm.}

As we show in the following theorem, \(T\) and thus \(\xi\) are weakly continuous within the class of \(\ell_1\)-norm symmetric distributions. 
To this end, denote by \(\SSS_d\) a \(d\)-variate random vector that is uniformly distributed on the unit simplex \(\cS_d = \{\xx\in \R^d \mid \lVert \xx \rVert_1 =1\}.\)
A \(d\)-variate random vector \(\WW\) follows an \emph{\(\ell_1\)-norm symmetric distribution} if there exists a nonnegative random variable \(R\) independent of \(\SSS_d\) such that \(\WW \eqd R\,\SSS_d.\)

\begin{theorem}[Continuity for \(\ell_1\)-norm symmetric distributions]\label{prpl1nsd}~\\
Let \((\YY_n,\XX_n)\eqd R_n \SSS_{q+p}\) be a sequence of \(\ell_1\)-norm symmetric random vectors and let \((\YY,\XX)\eqd R \,\SSS_{q+p}.\)
Assume that \(F_{R_n}\) and \(F_R\) are continuous with \(F_{R_n}(0)=F_R(0)=0\) for all \(n \in \mathbb{N}.\)
If \(R_n\xrightarrow{~d~} R\), then conditions \eqref{Cont.MainThm.Multi.12} - \eqref{Cont.MainThm.Multi.3} in Corollary \ref{Cont.MainThm.Multi} are satisfied and thus \(T(\YY_n,\XX_n) \to T(\YY,\XX).\)
\end{theorem}

The proof of Theorem \ref{prpl1nsd} (given in the appendix) is based on continuity properties of Archimedean copulas which are the copulas of \(\ell_1\)-norm symmetric distributions; see \cite{McNeil-2009}. 
For Archimedean copulas, weak convergence and weak conditional convergence are equivalent; see \citet[Theorem 4.1]{kasper2022}.
As a direct consequence, we obtain continuity of \(T\) within parametric Archimedean copula families such as the Clayton, the Gumbel-Hougaard or the Frank copula family.

%%%%%%%%%%%%%%%%%%%%%%%%%%%%%%%%%%%%%%%%%%%%%%%%%%%%%%%%%%%%%%%%%%%%%%%%%%%%%%%%%%%%%%%%
%%%%%%%%%%%%%%%%%%%%%%%%%%%%%%%%%%%%%%%%%%%%%%%%%%%%%%%%%%%%%%%%%%%%%%%%%%%%%%%%%%%%%%%%
%\clearpage

%\appendix
\section*{Appendix}

% Zähler für Theoreme auf 0 zurücksetzen
\setcounter{theorem}{0} 

% Das Aussehen der Nummerierung auf "A.1, A.2..." ändern
\renewcommand{\thetheorem}{A.\arabic{theorem}}

%%%%%%%%%%%%%%%%%%%%%%%%%%%%%%%%%%%%%%%%%%%%%%%%%%%%%%%%%%%%%%%%%%%%%%%%%%%%%%%%%%%%%%%%
%\section{Proofs from Section \ref{Sec.Cont.PM}}\label{appendix}

%%%%%%%%%%%%%%%%%%%%%%%%%%%%%%%%%%%%%%%%%%%%%%%%%%%%%%%%%%%%%%%%%%%%%%%%%%%%%%%%%%%%%%%%
%\subsection{Proof of Theorem \ref{theconell}}

For the proof of Theorem \ref{theconell}, we make use of the following result on continuous convergence of densities; see \citet[Theorem 2]{Sweeting-1986}.

\begin{lemma}\label{lemmswee}
    Let \((G_n)_{n\in \N}\) be a sequence of absolutely continuous (w.r.t. the Lebesgue measure) distributions on \((0,\infty)\) and let \((g_n)_{n\in \N}\) be a corresponding sequence of densities. Then the following two statements are equivalent:
    \begin{enumerate}[(i)]
        \item \label{lemmswee1} \((g_n)_{n\in \N}\) is asymptotically uniformly equicontinuous and bounded, and \(G_n\xrightarrow[]{d}G,\)
        \item \label{lemmswee2} \(g_n\to g\) uniformly in \(\R,\) where \(g\) is the uniformly continuous density of \(G.\)
    \end{enumerate}
\end{lemma}

%%%%%%%%%%%%%%%%%%%%%%%%%%%%%%%%%%%%%%%%%%%%%%%%%%%%%%%%%%%%%%%%%%%%%%%%%%%%%%%%%%%%%%%%
\begin{proof} [Proof of Theorem \ref{theconell}.]
Without loss of generality, let \(\boldsymbol{\mu}_n=\boldsymbol{\mu}=0\) for all \(n\in \N\).
We first prove Theorem \ref{theconell} under assumption \eqref{theconell1}.    
To verify condition \eqref{Cont.MainThm.Multi.1} of Corollary \ref{Cont.MainThm.Multi}, we obtain for the characteristic functions that
    \begin{align}\label{protheconell1}
        \varphi_{(\YY_n,\XX_n)}(\ttt) 
        = \phi (\ttt\Sigma_n \ttt^T) 
        \to \phi (\ttt\Sigma \ttt^T) 
        = \varphi_{(\YY,\XX)}(\ttt) ~~~\text{for all } \ttt\in \R^{q+p},
    \end{align}
where we use for the convergence that \(\phi\) is the characteristic function of a spherical distribution, thus being uniformly continuous.
Due to L\'{e}vy's continuity theorem (see, e.g.~\citet[Theorem 5.3]{Kallenberg-2002}), \eqref{protheconell1} then implies \((\YY_n,\XX_n)\xrightarrow{~d~}(\YY,\XX)\).
\\
In order to verify condition \eqref{Cont.MainThm.Multi.2} of Corollary \ref{Cont.MainThm.Multi}, we restrict ourselves to the case \(q=1\). 
The general case follows similarly, considering subvectors of the form \((Y_i,(Y_1,\ldots,Y_{i-1},\XX))\) and \((Y_i,(Y_1,\ldots,Y_{i-1})\), respectively, for \(i\in \{2,\ldots,q\}\).
As in the proof of Theorem \ref{Cont.MainLemma}, we use \citet[Lemma 3 \& Lemma 4]{Sethuraman-1961} in combination with the characterization of uniform conditional convergence in \citet[Theorem 4]{Sweeting_1989}.
Let \(V\) be the column space of \(\Sigma_{22}\) in \eqref{eqdecsig} without the null vector.
    %We show that, for all \(y\in \R,\) the conditional distribution functions \((F_{Y_n|\XX_n=\xx}(y))_{n\in \N}\) are asymptotically equicontinuous in \(\xx\in V\). 
    %First, assume that \(\Sigma\) is positive definite and thus invertible.
    %Then, 
    According to \citet[Corollary 5]{Cambanis-1981}, the conditional distributions admit the stochastic representations
    \begin{align}
    \begin{split}\label{eqtheconell1}
        \left(Y_n\mid \XX_n=\xx\right) 
        &\eqd \mu_{n,\xx} + R_{q_n(\xx)}\sigma_n^* \UU^{(1)},
        \\
        \left(Y\mid \XX = \xx\right) 
        &\eqd \mu_\xx + R_{q(\xx)}\sigma^* \UU^{(1)} ,
        \end{split}
    \end{align}
    where 
    \begin{align}
    \sigma_n^* 
    \nonumber &:= \Sigma_{n,11}-\Sigma_{n,12}\Sigma_{n,22}^{-1}\Sigma_{n,21},
    & \sigma^* 
    &:= \Sigma_{11}-\Sigma_{12}\Sigma_{22}^{-1}\Sigma_{21}, 
    \\
    \nonumber \mu_{n,\xx} 
    &:= \Sigma_{n,12} \Sigma_{n,22}^{-1}\xx, 
    &\mu_\xx
    &:= \Sigma_{12}\Sigma_{22}^{-1}\xx,
    \\
    \nonumber q_n(\xx)
    &:= \xx^T \Sigma_{n,22}^{-1} \xx, 
    & q(\xx)
    &:= \xx^T\Sigma_{22}^{-1} \xx,
    \\
   \label{ihidvij} R_{q_n(\xx)}
    &\eqd ((R^2-q_n(\xx))^{1/2} \, | \, \XX_n = \xx), 
    & R_{q(\xx)}
    &\eqd ((R^2-q(\xx))^{1/2} \, | \, \XX = \xx).
    \end{align}
    Recall that \(R\) in \eqref{ihidvij} is the radial part (as in \eqref{radialpart}), which, according to assumption \eqref{theconell1}, is equal for both \((\YY_n,\XX_n)\) and \((\YY,\XX)\).
    Here \(\Sigma_{n,11},\Sigma_{n,12},\Sigma_{n,21},\) and \(\Sigma_{n,22}\) are the submatrices of the decomposition of \(\Sigma_n\) similar to \eqref{eqdecsig}.
    From \(\Sigma_n \to \Sigma\) it then follows that
    \begin{align}
        \label{sigkon1} \sigma_n^* &\to \sigma^*,\\
        \label{sigkon2} \mu_{n,\xx} &\to \mu_\xx , ~~~~~ \text{where } \mu_\xx \text{ is continuous in } \xx,\\
        \label{sigkon3} q_n(\xx) &\to q(\xx), ~~~ \text{where } q(\xx) \text{ is continuous in } \xx,
    \end{align}
    using that \(\Sigma_{n,22}^{-1} \to \Sigma_{22}^{-1}\) due to the positive definiteness of \(\Sigma_n\) and \(\Sigma.\) 
    %By applying \citet[Corollary 5]{Cambanis-1981} a second time, now 
    Using the representation of the conditional radial distribution function \citet[Eq.~(15)]{Cambanis-1981}, we obtain for fixed \(r>0\) and for $\xx_n \to \xx$ that
     \begin{align}\label{aaba2}
        F_{R_{q_n(\xx_n)}}(r) &= \frac{\int_{(\sqrt{q_n(\xx_n)},\sqrt{r^2+q_n(\xx_n)}]}\left(s^2-   q_n(\xx_n)\right)^{-1/2}s^{-(p-1)} \de F_R(s)}{\int_{(\sqrt{q_n(\xx_n)},\infty)}\left(s^2-q_n(\xx_n)\right)^{-1/2}s^{-(p-1)} \de F_R(s)} \\
        \label{aaba21}\xrightarrow{n\to \infty} &\frac{\int_{(\sqrt{q(\xx)},\sqrt{r^2+q(\xx)}]}\left(s^2- q(\xx)\right)^{-1/2}s^{-(p-1)} \de F_R(s)}{\int_{(\sqrt{q(\xx)},\infty)}\left(s^2-q(\xx)\right)^{-1/2}s^{-(p-1)} \de F_R(s)} = F_{R_{q(\xx)}}(r),
    \end{align} 
    where the convergence is immediate from dominated convergence and the continuity of \(F_R\).
    Thus, $R_{q_n(\xx_n)} \xrightarrow[]{~d~} R_{q(\xx)}$  yields  $R_{q_n(\xx_n)} \, \sigma_n^* \, \UU^{(1)} \xrightarrow[]{~d~} R_{q(\xx)} \, \sigma^* \, \UU^{(1)}$ using Slutzky's theorem. Hence, 
    \begin{align*}
    \varphi_{Y_n|\XX_n=\xx_n}(t) 
    & = \varphi_{\mu_{n,\xx_n} + R_{q_n(\xx_n)} \, \sigma_n^* \, \UU^{(1)}}(t) 
     = \exp(it \, \mu_{n,\xx_n}) \, \varphi_{R_{q_n(\xx_n)} \, \sigma_n^* \, \UU^{(1)}}(t)
    \\
    & \xrightarrow[n\to \infty]{} \exp(it \, \mu_{\xx}) \, \varphi_{R_{q(\xx)} \, \sigma^* \, \UU^{(1)}}(t)
      = \varphi_{Y|\XX=\xx}(t)
    \end{align*}
    for all $t\in \R$, due to \eqref{eqtheconell1}.
    By a similar reasoning, we further have $\varphi_{Y|\XX=\xx_n}(t) 
    \xrightarrow[n\to \infty]{} \varphi_{Y|\XX=\xx}(t)$
    for all $t\in \R$.
    Now, applying \citet[Lemma 3 \& Lemma 4]{Sethuraman-1961} we obtain
    \begin{align}\label{theconell.Eq3}
    \int_{\R} f(y) \de P^{Y_n|\XX_n=\xx} 
    &\xrightarrow{~u~} \int_{\R} f(y) \de P^{Y|\XX=\xx} ~~~\text{on } V \text{ for all }f \in \cC_b(\R).
    \end{align}
    Using the characterization of uniform conditional convergence in \citet[Theorem 4]{Sweeting_1989} together with \(\XX_n\xrightarrow{~d~}\XX\) yields condition \eqref{Cont.MainThm.Multi.2} of Corollary \ref{Cont.MainThm.Multi}.\\
Recall that condition \eqref{Cont.MainThm.Multi.12} of Corollary \ref{Cont.MainThm.Multi} is automatically satisfied as a consequence of Theorem \ref{Cont.MainLemma}.
%%% Comments
%    The latter further implies continuity of \(\xx \mapsto F_{R_{q(\xx)}}\) and continuity of \(y \mapsto F_{Y|\XX=\xx}(y)\), both of which we will make use of right away.
%    The above convergence together with \eqref{sigkon1} yields \(R_{q_n(\xx)} \sigma_n^* U^{(1)} \xrightarrow{d} R_{q(\xx)}\sigma^* U^{(1)}\) which, together with \eqref{sigkon2} and the stochastic representation in \eqref{eqtheconell1}, then gives \(F_{Y_n|\XX_n=\xx}(y) \to F_{Y|\XX=\xx}(y)\) for all \(y\in \R\) and \(\xx\in V\). 
%    Using the continuity of \(\xx \mapsto F_{Y|\XX=\xx}(y)\) (due to the continuity of \(\xx \mapsto F_{R_{q(\xx)}}\) and \(\xx \mapsto \mu_\xx\)) and the continuous convergence of \((F_{Y_n|\XX_n=\xx}(y))_{n\in \N}\) (due to \eqref{eqtheconell1} - \eqref{aaba21}), we can conclude that the conditional distribution functions \((F_{Y_n|\XX_n=\xx}(y))_{n\in \N}\) are asymptotically equicontinuous in \(\xx\in V\) for all \(y\in \R\). \\
\\ 
Since \(F_{Y_{i,n}}\) and \(F_{Y_i}\) are continuous for all \(i\) and \(n,\) condition \eqref{Cont.MainThm.Multi.3} of Corollary \ref{Cont.MainThm.Multi} is trivially satisfied.
This proves the first part of Theorem \ref{theconell}.

For the proof of Theorem \ref{theconell} under assumption \eqref{theconell2}, we extend the above proof under assumption \eqref{theconell1} to the case where the radial variables depend on \(n\). 
To verify condition \eqref{Cont.MainThm.Multi.1} of Corollary \ref{Cont.MainThm.Multi}, we obtain for the characteristic functions that
\begin{align*}
  \varphi_{(\YY_n,\XX_n)}(\ttt) 
  = \phi_n (\ttt\Sigma_n \ttt^T) \to \phi (\ttt\Sigma \ttt^T) 
  = \varphi_{(\YY,\XX)}(\ttt) ~~~\text{for all } \ttt\in \R^{q+p},
\end{align*}
where we use for the convergence that \(\phi\) and \((\phi_n)_{n \in \mathbb{N}}\) are the characteristic functions of spherical distributions with \(\phi_n \to \phi\) pointwise by assumption. L\'{e}vy's continuity theorem (see, e.g.~\citet[Theorem 5.3]{Kallenberg-2002}) then implies \(\phi_n \to \phi\) uniformly on compact sets, hence  \(\phi_n \to \phi\) continuously. Applying L\'{e}vy's continuity theorem a second time then gives \((\YY_n,\XX_n)\xrightarrow{~d~}(\YY,\XX)\).
\\
We now verify condition \eqref{Cont.MainThm.Multi.2} of Corollary \ref{Cont.MainThm.Multi} and, again, restrict ourselves to the case \(q=1\).
Notice that the stochastic representation of the conditional distribution \(Y_n\mid \XX_n=\xx\) in \eqref{eqtheconell1} modifies to
\begin{align}\label{eqtheconell2}
    (Y_n\mid \XX_n=\xx) \eqd \mu_{n,\xx} + R_{n,q_n(\xx)}  \sigma_n^* \UU^{(1)}
\end{align}
where \(R_{n,q_n(\xx)}\eqd ((R_n^2-q_n(\xx))^{1/2}\mid \XX_n=\xx),\) \(n\in \N.\)
As before, we prove the assertion using \citet[Lemma 3 \& Lemma 4]{Sethuraman-1961} in combination with the characterization of uniform conditional convergence in \citet[Theorem 4]{Sweeting_1989}.
Therefore, we need to show that \(R_{n,q_n(\xx_n)} \xrightarrow[]{d}R_{q(\xx)}\) for all \(\xx_n, \xx\in V\) such that \(\xx_n \to \xx\), where \(V\) are the inner points of the closed support of \(\XX.\) 
To do so, we verify that the associated distribution functions converge pointwise.
For \(a>0\), consider the modified denominator of the right-hand side in \eqref{aaba2} given by
\begin{align*}
    G_n(a^2)&:=\int_{(a,\infty)} (s^2-a^2)^{-1/2} s^{-(p-1)} \de F_{R_n}(s)\\
    &\phantom{:}= \int_{(a,\infty)} (s^2-a^2)^{-1/2} s^{-(p-1)} f_n(s) \de s
    \phantom{:}= \int_{(0,\infty)} (z^2+a^2)^{-p/2} f_n(\sqrt{z^2+a^2}) \de z,
\end{align*}
where we apply that \(f_n\) is the density of \(R_n\) and where the second equality follows from the transformation \(z^2:=s^2-a^2.\)
Using Lemma \ref{lemmswee}, we obtain from the assumptions that also \(R\) has a density which we denote by \(f.\) We denote the denominator of \eqref{aaba21} as
\begin{align*}
    G(a^2)&:=\int_{(a,\infty)} (s^2-a^2)^{-1/2} s^{-(p-1)} \de F_{R}(s) = \int_{(0,\infty)} (z^2+a^2)^{-p/2} f(\sqrt{z^2+a^2}) \de z.
\end{align*}
Similarly, the modified numerators of \eqref{aaba2} and \eqref{aaba21} are
\begin{align*}
    H_n(a^2) &:= \int_{(a,\sqrt{r^2+a^2})} (s^2-a^2)^{-1/2} s^{-(p-1)} \de F_{R_n}(s) = \int_{(0,r)} (z^2+a^2)^{-p/2} f_n(\sqrt{z^2+a^2}) \de z,\\
    H(a^2) &:= \int_{(a,\sqrt{r^2+a^2})} (s^2-a^2)^{-1/2} s^{-(p-1)} \de F_{R}(s) = \int_{(0,r)} (z^2+a^2)^{-p/2} f(\sqrt{z^2+a^2}) \de z.
\end{align*}
We aim to prove that 
\begin{align}\label{prognqn}
    G_n(q_n(\xx_n)) \to G(q(\xx))~~~ \text{for all } \xx_n, \xx \in V  \text{ with } \xx_n \to \xx.
\end{align}
Similarly, \(H_n(q_n(\xx_n)) \to H(q(\xx))\) and thus \(F_{R_{n,q_n(\xx_n)}}(r) \to F_{R_{q(\xx)}}(r)\) for all \(\xx_n,\xx\in V\) with \(\xx_n \to \xx\) and for all \(r>0.\)
\\
To prove \eqref{prognqn}, we first show that 
\begin{align}\label{prognqn1}
    G_n\to G ~~~ \text{uniformly on } [K_1,K_2]
\end{align}
for any \(0 < K_1 < K_2.\)
Therefore, let \(\varepsilon>0\) and \(a^2 \in [K_1,K_2].\) 
Since \(f_n\to f\) pointwise due to Lemma \ref{lemmswee}, applying Scheff\'{e}s Lemma, there exists \(n_0\in \N\) such that 
\begin{align}
    \left|\int_{(r,\infty)} \underbrace{(z^2+a^2)^{-{p/2}}}_{\leq r^{-p}}  \left(f_n(\sqrt{z^2+a^2})  -  f(\sqrt{z^2+a^2})\right) \de z \right| < \varepsilon
\end{align}
for all \(n\geq n_0.\) Further, due to Lemma \ref{lemmswee}, there exists \(n_1\in \N\) such that \(\sup_{y>0} |f_n(y)-f(y)|<\varepsilon\) for all \(n\geq n_1\).
%\begin{align}\label{eegge1}
%    \sup_{y>0} |f_n(y)-f(y)|<\varepsilon ~~~\text{for all } n\geq n_1.
%\end{align} 
%If \(y\mapsto y^{-1} \sup_n \{f_n(y)\} = y^{-1} M(y)\) is integrable on \((r,\infty),\) applying the dominated convergence theorem, there exists \(n_1\in \N\) such that
%for all \(p\geq 1,\) using that \((z^2+a^2)^{-p/2} \leq z^{-p}.\)
%and applying Lemma \ref{lemmaxoj} where \(M(y)\to 0\) as \(y\to \infty.\) 
It follows that
\begin{align*}
    \left| G_n(a^2)-G(a^2)\right| &\leq \int_{(0,r)} \underbrace{(z^2+a^2)^{-{p/2}}}_{ < a^{-p}}\underbrace{\left| f_n(\sqrt{z^2+a^2})-f(\sqrt{z^2+a^2})\right|}_{<\varepsilon} \de z \\
   &  + \underbrace{\left|\int_{(r,\infty)} (z^2+a^2)^{-{p/2}}  \left(f_n(\sqrt{z^2+a^2})  -  f(\sqrt{z^2+a^2})\right) \de z \right|}_{~~~~~~~~~~~~~ < \varepsilon \text{ for all } n \geq n_1} \\
   & <  (a^{-p} \, r  + 1 ) \,\varepsilon 
   \leq (K_1^{-p} \, r  + 1 ) \,\varepsilon ,
\end{align*}
for all \(n\geq \max\{n_0,n_1\},\) which proves \eqref{prognqn1}.
Now, let \(\xx_n,\xx\in V\) with \(\xx_n \to \xx\). 
Then the sequence \((q_n(\xx_n))_n\) converges by \eqref{sigkon3} to \(q(\xx).\) 
Since \(\xx_n, \xx \ne \boldsymbol{0}\) and \(\Sigma\) as well as \(\Sigma_n\) are positive definite, we have \(q(\xx)>0\) and \(q_n(\xx_n)>0\) for all \(n.\) Hence, there exist \(0<L_1<L_2\) such that \(q_n(\xx_n),q(\xx)\in [L_1,L_2]\) for all \(n.\) 
%Now, let \(\varepsilon>0.\) 
Since \((G_n)_{n\in \N}\) converges uniformly on \([L_1,L_2],\) there exists \(N_1\in \N\) such that \(|G_n(a^2)-G(a^2)|<\varepsilon/2\) for all \(n\geq N_1\) and for all \(a^2 \in [L_1,L_2]\). 
Since $G$ is continuous and \(q_n(\xx_n)\to q(\xx),\) there exists \(N_2\in \N\) such that \(|G(q_n(\xx_n)) - G(q(\xx))| < \varepsilon/2\) for all \(n\geq N_2.\) 
It follows that
\begin{align*}
    |G_n(q_n(\xx_n)) - G(q(\xx)) | &\leq |G_n(q_n(\xx_n)) -G(q_n(\xx_n))|  + | G(q_n(\xx_n)) - G(q(\xx)) | 
    < \varepsilon/2 + \varepsilon/2 =  \varepsilon 
\end{align*}
for all \(n\geq \max\{N_1,N_2\},\) which proves \eqref{prognqn}.
Since $R_{n,q_n(\xx_n)} \xrightarrow[]{~d~} R_{q(\xx)}$, we obtain with Slutzky's theorem that $R_{n,q_n(\xx_n)} \, \sigma_n^* \, \UU^{(1)} \xrightarrow[]{~d~} R_{q(\xx)} \, \sigma^* \, \UU^{(1)}$. The latter implies 
\begin{align*}
  \varphi_{Y_n|\XX_n=\xx_n}(t) 
  & = \varphi_{\mu_{n,\xx_n} + R_{n,q_n(\xx_n)} \, \sigma_n^* \, \UU^{(1)}}(t) 
 = \exp(it \, \mu_{n,\xx_n}) \, \varphi_{R_{n,q_n(\xx_n)} \, \sigma_n^* \, \UU^{(1)}}(t)
  \\
  & \xrightarrow[n\to \infty]{} \exp(it \, \mu_{\xx}) \, \varphi_{R_{q(\xx)} \, \sigma^* \, \UU^{(1)}}(t)
    = \varphi_{Y|\XX=\xx}(t)
\end{align*}
for all $t\in \R$, due to \eqref{eqtheconell1} and \eqref{eqtheconell2} .
By a similar reasoning, we have $\varphi_{Y|\XX=\xx_n}(t) 
\xrightarrow[n\to \infty]{} \varphi_{Y|\XX=\xx}(t)$
for all $t\in \R$.
Now, applying \citet[Lemma 3 \& Lemma 4]{Sethuraman-1961} we obtain
\begin{align*}
  \int_{\R} f(y) \de P^{Y_n|\XX_n=\xx} 
  &\xrightarrow{~u~} \int_{\R} f(y) \de P^{Y|\XX=\xx} ~~~\text{on } V \text{ for all }f \in \cC_b(\R).
\end{align*}
Using the characterization of uniform conditional convergence in \citet[Theorem 4]{Sweeting_1989} together with \(\XX_n\xrightarrow{~d~}\XX\) yields condition \eqref{Cont.MainThm.Multi.2} of Corollary \ref{Cont.MainThm.Multi}.
\\
Since \(F_{Y_{i,n}}\) and \(F_{Y_i}\) are continuous for all \(i\) and \(n,\) condition \eqref{Cont.MainThm.Multi.3} of Corollary \ref{Cont.MainThm.Multi} is trivially satisfied, which proves the second part of Theorem \ref{theconell}.
\end{proof}

%%%%%%%%%%%%%%%%%%%%%%%%%%%%%%%%%%%%%%%%%%%%%%%%%%%%%%%%%%%%%%%%%%%%%%%%%%%%%%%%%%%%%%%%

\begin{proof}[Proof of Proposition \ref{thenormodextcas}.]
\eqref{thenormodextcas1}
Due to \citet[Theorem 2.1]{ansari2023MFOCI}, \(\YY\) and \(\XX\) are independent if and only if \(T(\YY,\XX)=0.\) As a property of elliptical distributions, \(\YY\) and \(\XX\) can only be non-degenerate independent if they follow a normal distribution; see e.g. \citet[Section 5(d)]{Cambanis-1981}.
%Note that \(\XX\) and \(\YY\) are non-degenerate because \(\sigma_{ii}>0\) for all \(i.\)
Hence, the assertion follows from the well-known property of multivariate normal distributions that \(\YY\) and \(\XX\) are independent if and only if \(\Sigma_{21}\) is the null matrix, see e.g.~\citet[Corollary 2 in Section 2.3]{FangKotz-1990}.\\
To show statement \eqref{thenormodextcas2},
consider the decomposition \({\boldsymbol \mu}=({\boldsymbol \mu}_1,{\boldsymbol \mu}_2),\) \({\boldsymbol \mu}_1\in \R^q,\) \({\boldsymbol \mu}_2\in \R^p\) and define \(k:=\rank(\Sigma)-\rank(\Sigma_{22})\geq 0.\) Then it holds that
%\begin{align}\label{eqthenormodextcas2}
    \((\YY\mid \XX=\xx) \sim \cE({\boldsymbol \mu}_\xx,\Sigma^*,\phi_\xx)\)
%\end{align}
with a stochastic representation
\begin{align}\label{repprofsell}
    (\YY \mid \XX=\xx)\eqd {\boldsymbol \mu}_\xx + \ZZ
\end{align}
for \({\boldsymbol \mu}_\xx={\boldsymbol \mu}_1+(\xx-{\boldsymbol \mu}_2)\Sigma_{22}^-\Sigma_{21}\) and 
\(\Sigma^*:=\Sigma_{11}-\Sigma_{12}\Sigma_{22}^-\Sigma_{21},\) where \(\rank(\Sigma^*)=k\) and where \(\ZZ\) is a \(q\)-dimensional \(\cE( {\bf{0}}, \Sigma^*,\phi_\xx)\)-distributed random vector with generator \(\phi_\xx\) depending on \(\xx,\) \({\boldsymbol \mu}_2\) and \(\Sigma_{22}^-\); see e.g.~\citet[Corollary 5]{Cambanis-1981}. Here \(A^-\) denotes a generalized inverse of a symmetric matrix \(A\) with positive rank. It follows by the characterization of perfect dependence \citet[Theorem 2.1]{ansari2023MFOCI} that
\begin{align*}
&&T(\YY,\XX) &=1 \\
&\Longleftrightarrow  &\YY &= f(\XX) ~~~\text{a.s.} \\
&\Longleftrightarrow &\YY &= {\boldsymbol \mu}_1 + (\XX-{\boldsymbol \mu}_2) \Sigma_{22}^-\Sigma_{21} ~~~ \text{a.s.} \\
&\Longleftrightarrow & k&=0\\
&\Longleftrightarrow &\rank(\Sigma)&=\rank(\Sigma_{22}),
\end{align*}
where the second equivalence follows with \eqref{repprofsell}. For the third equality, we observe that \(\ZZ=\bold{0}\) almost surely if and only if \(\rank(\Sigma^*)=0.\) Finally, the last equality holds true by the definition of \(k.\)
%\hfill \qedsymbol
\end{proof}

%%%%%%%%%%%%%%%%%%%%%%%%%%%%%%%%%%%%%%%%%%%%%%%%%%%%%%%%%%%%%%%%%%%%%%%%%%%%%%%%%%%%%%%%
\begin{proof}[Proof of Theorem \ref{prpl1nsd}.]
Since \(\ell_1\)-norm symmetric distributions are closed under marginalization \citet[Theorem 3.1]{McNeil-2009}, we may restrict ourselves to the case \(q=1.\)
The general case follows similarly, considering subvectors of the form \((Y_i,(Y_1,\ldots,Y_{i-1},\XX))\) and \((Y_i,(Y_1,\ldots,Y_{i-1})\), respectively, for \(i\in \{2,\ldots,q\}\).
Let \((Y_n,\XX_n)\eqd R_n \, \SSS_{1+p}\) be a sequence of \(\ell_1\)-norm symmetric random vectors and let \((Y,\XX)\eqd R \, \SSS_{1+p}\) such that \(F_{R_n}\) and \(F_R\) are continuous with \(F_{R_n}(0)=F_R(0)=0\) for all \(n \in \mathbb{N}.\)
Then, the random quantities \(Y\), \(Y_n\), \(\XX\), and \(\XX_n\) have continuous distribution functions.
%In what follows we use the characterization of \(\ell_1\)-norm symmetric distributions by Archimedean copulas due to \cite{McNeil-2009}.
Independence of \(R_n\) and \(\SSS_{1+p}\) (and similarly of \(R\) and \(\SSS_{1+p}\)) implies for the characteristic functions that
\begin{align*}
  \varphi_{(Y_n,\XX_n)}({\bf t}) 
  & = \varphi_{R_n\SSS_{1+p}}({\bf t})
    = \int_{[0,\infty)} \mathbb{E} \left( \exp(i\,{\bf t}' (r\SSS_{1+p})) \right) \de P^{R_n}(r)
    = \int_{[0,\infty)} \varphi_{{\bf t}'\SSS_{1+p}} (r) \de P^{R_n}(r).
\end{align*}
Since \(r \mapsto \varphi_{{\bf t}'\SSS_{1+p}} (r)\) is continuous and bounded, weak convergence \(R_n\xrightarrow{~d~} R\) implies \(\varphi_{(Y_n,\XX_n)} = \varphi_{R_n\SSS_{1+p}} \to \varphi_{R\SSS_{1+p}} = \varphi_{(Y,\XX)}\) pointwise and thus \((Y_n,\XX_n)\xrightarrow{d} (Y,\XX)\).
Hence, by \citet[Lemma A.2]{sfx2021vine}, the underlying copulas converge uniformly, i.e., \(C_{(Y_n,\XX_n)} \to C_{(Y,\XX)}\) uniformly, and similarly, for the associated survival copulas, \(\hat{C}_{(Y_n,\XX_n)} \to \hat{C}_{(Y,\XX)}\) uniformly.
Due to \citet[Theorem 3.1(i)]{McNeil-2009}, \((Y_n,\XX_n)_{n \in \mathbb{N}}\) and \((Y,\XX)\) have Archimedean survival copulas \((\hat{C}_{(Y_n,\XX_n)})_{n \in \mathbb{N}}\) and \(\hat{C}_{(Y,\XX)}\)
that are the distribution functions of the random vectors \((1-F_{Y_n}(Y_n),{\bf 1}-{\bf F}_{\XX_n}(\XX_n))_{n \in \mathbb{N}}\) and \((1-F_Y(Y),{\bf 1}-{\bf F}_{\XX}(\XX))\) respectively, where \({\bf 1} := (1,\ldots,1)\in \R^p \) and \({\bf F}_{\XX_n} = (F_{X_{n,1}},\ldots,F_{X_{n,p}}),\) similarly for \({\bf F}_{\XX}.\)
%\((\psi_n)_{n\in \N}\) is uniformly bounded and since \(\psi_n\) and \(\psi\) are convex and continuous with \(\psi_n(0)=\psi(0)=1,\) it follows that \(\psi_n\) converges to \(\psi\) uniformly due to Arzel{\`a}-Ascoli's theorem.
By \citet[Theorem 4.1]{kasper2022}, 
we obtain on the one hand that
the sequence of densities \((\hat{c}_{\XX_n})_{n \in \mathbb{N}}\) of copulas \(\hat{C}_{\XX_n}\), \(n \in \mathbb{N}\), converges \(\lambda^p\)-almost everywhere in \([0,1]^p\) to the density \(\hat{c}_{\XX}\) of copula \(\hat{C}_{\XX}\). On the other hand, 
there exists some measurable set \(\Lambda \subseteq [0,1]^p\) with \(\mu_{\hat{C}_{\XX}}(\Lambda)=1\) (where \(\mu_C\) denotes the copula measure of the copula \(C\)) such that for every \(\uu \in \Lambda\) the sequence of Markov kernels (i.e.~the regular conditional distributions) \((K_{\hat{C}_{(Y_n,\XX_n)}}(\uu,.))_{n \in \mathbb{N}}\) of copulas \(\hat{C}_{(Y_n,\XX_n)}\), \(n \in \mathbb{N}\), converges weakly to the Markov kernel \(K_{\hat{C}_{(Y,\XX)}}(\uu,.)\) of the copula \(\hat{C}_{(Y,\XX)}\).
Since $F_R$ is assumed to be continuous, \citet[Theorem 5.12]{kasper2022} implies that the limiting Markov kernels \(t \mapsto K_{\hat{C}_{(Y,\XX)}}(\uu,[0,t])\) are continuous distribution functions.
With the above considerations, it follows for \(k_n(\uu,s) := K_{\hat{C}_{(Y_n,\XX_n)}}(\uu,[0,s])\) and \(k(\uu,s) := K_{\hat{C}_{(Y,\XX)}}(\uu,[0,s])\) that
\begin{eqnarray*}
  \lefteqn{ \left| P \big(1-F_{Y_n}(Y_n) \leq t, 1-F_{Y_n}(Y_n') \leq t'\big) - P \big(1-F_{Y}(Y) \leq t, 1-F_{Y}(Y') \leq t'\big) \right|} 
  \\
  &   =  & \Big| \int_{[0,1]^p} \hspace{-.5cm} P \big(1-F_{Y_n}(Y_n) \leq t |  {\bf 1}-{\bf F}_{\XX_n}(\XX_n) = \uu \big)  \, 
  P \big(1-F_{Y_n}(Y_n') \leq t' |  {\bf 1}-{\bf F}_{\XX_n}(\XX_n) = \uu \big) \de P^{{\bf 1}-{\bf F}_{\XX_n}(\XX_n)} (\uu) 
  \\
  &      & \qquad - \int_{[0,1]^p} \hspace{-.5cm} P \big(1-F_{Y}(Y) \leq t \mid  {\bf 1}-{\bf F}_{\XX}(\XX) = \uu \big) \,
  P \big(1-F_{Y}(Y') \leq t' \mid  {\bf 1}-{\bf F}_{\XX}(\XX) = \uu \big) \de P^{{\bf 1}-{\bf F}_{\XX}(\XX)} (\uu) \Big| 
  \\
  &   =  & \Big| \int_{[0,1]^p} k_n(\uu,t) \, k_n(\uu,t') \; \de \mu_{\hat{C}_{\XX_n}}(\uu)
 - \int_{[0,1]^p} k(\uu,t) \, k(\uu,t') \; \de \mu_{\hat{C}_{\XX}}(\uu) \Big|
  \\
  &   =  & \Big| \int_{[0,1]^p} k_n(\uu,t) \, k_n(\uu,t') \; \de \mu_{\hat{C}_{\XX_n}}(\uu)
   - \int_{[0,1]^p} k_n(\uu,t) \, k_n(\uu,t') \; \de \mu_{\hat{C}_{\XX}}(\uu) \Big|
  \\
  &      & + \; \Big| \int_{[0,1]^p} k_n(\uu,t) \, k_n(\uu,t') \; \de \mu_{\hat{C}_{\XX}}(\uu)
  - \int_{[0,1]^p} k(\uu,t) \, k(\uu,t') \; \de \mu_{\hat{C}_{\XX}}(\uu) \Big|
  \\
  & \leq & \int_{[0,1]^p} k_n(\uu,t) \, k_n(\uu,t') \, \left| \hat{c}_{\XX_n} (\uu) - \hat{c}_{\XX} (\uu) \right| \; \de \lambda^p (\uu)
  \\
  &      & + \; \int_{[0,1]^p} \left| k_n(\uu,t) - k(\uu,t) \right| \; \de \mu_{\hat{C}_{\XX}}(\uu)
   + \; \int_{[0,1]^p} \left| k_n(\uu,t') - k(\uu,t') \right| \; \de \mu_{\hat{C}_{\XX}}(\uu)
  \\
  & \longrightarrow & 0 \qquad \text{ when } {n \to \infty}, 
\end{eqnarray*}
where we use for the inequality that \(|ab-cd|=|b(a-c)+c(b-d)| \leq |a-c|+|b-d|\) for \(|b|,|c| \leq 1\).
The convergence follows from the pointwise convergence of the densities and from the weak convergence of the Markov kernels on the set \(\Lambda\).
Thus, the Markov products \((1-F_{Y_n}(Y_n),1-F_{Y_n}(Y'_n))\) relative to \({\bf 1}-{\bf F}_{\XX_n}(\XX_n)\) 
converge weakly to the Markov product \((1-F_{Y}(Y),1-F_{Y}(Y'))\) relative to 
\({\bf 1}-{\bf F}_{\XX}(\XX).\)
Hence, also 
%By disintegration, change of coordinates and the continuous mapping theorem, it then follows that also 
the sequence of Markov products \((F_{Y_n}(Y_n),F_{Y_n}(Y'_n))\) relative to \(\XX_n\) converges weakly to the Markov product \((F_{Y}(Y),F_{Y}(Y'))\) relative to \(\XX,\) i.e.,~the sequence of copulas \((C_{Y_n,Y'_n})_{n \in \mathbb{N}}\) converges uniformly to the copula \(C_{Y,Y'}\).
Finally, weak convergence of \((Y_n,Y_n')_{n \in \mathbb{N}}\) to \((Y,Y')\) follows from \citet[Lemma A.2(1)]{sfx2021vine} and the fact that \((F_{Y_n})_{n \in \mathbb{N}}\) weakly converges to the continuous distribution function \(F_Y\) by assumption.
\end{proof}

%Appendices should be provided in \verb|{appendix}| environment,
%before Acknowledgements.

%If there is only one appendix,
%then please refer to it in text as \ldots\ in the \hyperref[appn]{Appendix}.

%%%%%%%%%%%%%%%%%%%%%%%%%%%%%%%%%%%%%%%%%%%%%%
%% Example with multiple Appendixes:        %%
%%%%%%%%%%%%%%%%%%%%%%%%%%%%%%%%%%%%%%%%%%%%%%

%%%%%%%%%%%%%%%%%%%%%%%%%%%%%%%%%%%%%%%%%%%%%%
%% Support information, if any,             %%
%% should be provided in the                %%
%% Acknowledgements section.                %%
%%%%%%%%%%%%%%%%%%%%%%%%%%%%%%%%%%%%%%%%%%%%%%
\section*{Acknowledgments}
The authors would like to thank the anonymous referees and an Associate
Editor for detailed reading of the manuscript and their constructive comments that improved the
quality of this paper.

%%%%%%%%%%%%%%%%%%%%%%%%%%%%%%%%%%%%%%%%%%%%%%
%% Funding information, if any,             %%
%% should be provided in the                %%
%% funding section.                         %%
%%%%%%%%%%%%%%%%%%%%%%%%%%%%%%%%%%%%%%%%%%%%%%
\section*{Funding}
This research was funded in whole by the Austrian Science Fund (FWF) [10.55776/P36155] project \emph{ReDim: Quantifying Dependence via Dimension Reduction} and [10.55776/PAT1669224] project \emph{SORT: Stochastic orders for functional dependence}.
%This research was funded in whole by the Austrian Science Fund (FWF) [10.55776/P36155] project ReDim: Quantifying Dependence via Dimension Reduction. 
The authors further acknowledge the support from the WISS 2025 project `IDA-lab Salzburg' 20204-WISS/225/197-2019 and 20102-F1901166-KZP.

%%%%%%%%%%%%%%%%%%%%%%%%%%%%%%%%%%%%%%%%%%%%%%
%% Supplementary Material, including data   %%
%% sets and code, should be provided in     %%
%% {supplement} environment with title      %%
%% and short description. It cannot be      %%
%% available exclusively as external link.  %%
%% All Supplementary Material must be       %%
%% available to the reader on Project       %%
%% Euclid with the published article.       %%
%%%%%%%%%%%%%%%%%%%%%%%%%%%%%%%%%%%%%%%%%%%%%%

%%%%%%%%%%%%%%%%%%%%%%%%%%%%%%%%%%%%%%%%%%%%%%%%%%%%%%%%%%%%%
%%                  The Bibliography                       %%
%%                                                         %%
%%  imsart-???.bst  will be used to                        %%
%%  create a .BBL file for submission.                     %%
%%                                                         %%
%%  Note that the displayed Bibliography will not          %%
%%  necessarily be rendered by Latex exactly as specified  %%
%%  in the online Instructions for Authors.                %%
%%                                                         %%
%%  MR numbers will be added by VTeX.                      %%
%%                                                         %%
%%  Use \cite{...} to cite references in text.             %%
%%                                                         %%
%%%%%%%%%%%%%%%%%%%%%%%%%%%%%%%%%%%%%%%%%%%%%%%%%%%%%%%%%%%%%

%% if your bibliography is in bibtex format, uncomment commands:
\bibliographystyle{abbrvnat} % Style BST file (imsart-number.bst or imsart-nameyear.bst)
\bibliography{arxiv_continuity_xi.bib}       % Bibliography file (usually '*.bib')

\end{document}